\documentclass[a4,11pt]{amsart}
\usepackage[margin=3.5cm]{geometry}
\usepackage[utf8]{inputenc}
\usepackage{tikz}
\usepackage{xcolor}
\usepackage{verbatim}
\usepackage{amsmath,amsthm,amssymb}
\usepackage{url}
\usepackage{subcaption}
\usepackage{graphicx} 
\usepackage{ulem}
\usetikzlibrary{arrows.meta,
                chains,
                decorations.pathreplacing,calligraphy,
                positioning,
                quotes,
                shapes.geometric
                }

\newtheorem{theorem}{Theorem}[section]

\newtheorem{proposition}[theorem]{Proposition}
\newtheorem{lemma}[theorem]{Lemma}

\newtheorem{conjecture}[theorem]{Conjecture}
\newtheorem{example}[theorem]{Example}

\usepackage{tikz-cd} 

\usepackage{tikz-cd}
\usepackage{soul}
\usepackage{tkz-berge}

\definecolor{green}{HTML}{37cc57}
\definecolor{red}{HTML}{e81e5b}
\definecolor{blue}{HTML}{3d8bbf}

\DeclareMathOperator{\Aut}{Aut}

\DeclareMathOperator{\PSL}{PSL}

\DeclareMathOperator{\Cor}{Cor}

\DeclareMathOperator{\PG}{PG}

\DeclareMathOperator{\Alt}{A}

\DeclareMathOperator{\Sv}{S}
\DeclareMathOperator{\Sn}{S_n}

\DeclareMathOperator{\Mo}{M_{11}}
\DeclareMathOperator{\Md}{M_{12}}
\DeclareMathOperator{\Mvd}{M_{22}}
\DeclareMathOperator{\Mvt}{M_{23}}
\DeclareMathOperator{\Mvq}{M_{24}}

\DeclareMathOperator{\Sym}{Sym}
\DeclareMathOperator{\PR}{PR}

\DeclareMathOperator{\lcm}{lcm}

{}

\title[Geometries admitting trialities for $\Sn$ and $\Alt_n$]{Geometries admitting trialities for the symmetric and alternating groups}
\author{Remi Delaby}
\address{Remi Delaby, Universit\'e Libre de Bruxelles, D\'epartement de Math\'ematique, C.P.216 - Alg\`ebre et Combinatoire, Boulevard du Triomphe, 1050 Brussels, Belgium.}
\email{remi.delaby@ulb.be}

\date{\today}

\begin{document}

\begin{abstract}
In incidence geometry, a triality is a symmetry cyclically exchanging triples of types of elements. Requiring geometries with trialities to satisfy standard regularity conditions makes their construction highly non trivial, and known examples are rare. In this paper, we present two infinite families of flag transitive, thin and residually connected geometries admitting trialities and no dualities with type preserving automorphism groups isomorphic to $\Sv_n$ or $\Alt_n$. We also develop general methods that extend to other settings. The residues of the two families are fundamentally different. In the first family, the maximal parabolics have constant size while in the other their size grows essentially linearly with the degree of the group.

\end{abstract}

\maketitle

\section{Introduction}

Incidence geometries are structures consisting of elements of various types together with a specified incidence relation linking them. Classical examples include configurations of lines and points, or higher dimensional generalisations such as polytopes. In some cases, such geometries admit additional symmetries that permute the types of elements. Notable examples are, among others, projective duality which interchanges points and hyperplanes and polytope duality. A triality is a higher‑order analogue, in which three distinct element types are cyclically exchanged. The earliest occurence of this idea can be traced back to the work of Study (see~\cite[Page 435]{Porteous}, see also~\cite{Study1} and~\cite{Study2}\footnote{See \url{http://neo-classical-physics.info/uploads/3/4/3/6/34363841/study-analytical_kinematics.pdf} 
for an english translation of~\cite{Study2}.}) who used a quadric in seven dimensional space to describe motion. 
The phenomenon was later given the name triality by Cartan \cite{Cartan} in the context of Lie groups. The associated quadric was subsequently studied in detail by Tits \cite{tits1959trialite}. 

The notion of triality also appears in the study of maps on surfaces in which a triality can be obtained as a composition of the dual operator and the Petrie dual operator (see \cite{abrams2022new,jones2010maps,wilson1979operators}) on maps. Hypermaps, a generalisation of maps, also saw the idea of triality being explored (see \cite{jones2010hypermap}).

In recent years, interest has increased in incidence geometries that admit trialities but lack dualities. Among highly symmetric incidence geometries, these conditions impose strong structural constraints, and examples are scarce. The discovery of examples of such geometries is thus a necessary step for further understanding the nature of trialities. In \cite{LeemansStokes2019}, Leemans and Stokes developed constructions of coset geometries exhibiting trialities by considering groups with outer automorphisms of order three. The techniques were further refined in \cite{leemans2022incidence}, \cite{TrialitySuzuki}, and \cite{lineartri}, ultimately leading to the discovery of infinite families of flag‑transitive, residually connected geometries that admit trialities but no dualities for the Suzuki groups (in \cite{TrialitySuzuki}) and for the affine general linear groups (in \cite{lineartri}).

Whenever $n < 10$, exhaustive searches with {\sc Magma}~\cite{magma} show that there are no such thin geometries admitting $\Alt_n$ or $\Sv_n$ as type preserving automorphism group. Strikingly, however, these geometries become seemingly more and more common as $n$ grows larger. In this paper, we prove that there are infinitely many and provide two concrete families. For one of them, the size of the maximal parabolics is constant while for the other one, the size increases essentially linearly with $n$.

Programs were written in {\sc Magma} in order to illustrate and verify some of the constructions discussed in this work, including the aforementioned families. The relevant scripts are provided\footnote{\url{https://github.com/RDelaby/GeometriesAdmittingTrialitiesForSymAndAlt}} so that the readers can reproduce the examples and experiment on their own.

\section{preliminaries}

Geometric structures are composed of objects and a relation that specifies how these objects are related to one another. The notion of incidence system formalizes this concept by providing an abstract framework to study these configurations. We refer to~\cite{buekenhout2013diagram} for a more detailed introduction to this subject and more generally to the subject of diagram geometry.

Let $I$ be a finite non empty set.
    A triple $\Gamma = (X,\star,t)$ is called an \textit{incidence system} over $I$ if
    \begin{enumerate}
        \item $X$ is a non empty set whose elements are called the \textit{elements} of $\Gamma$;
        \item $t$ is a map from $X$ to $I$, called the \textit{type map} of $\Gamma$. Elements of $t^{-1}(i)$ are called elements of type $i$;
        \item $\star$ is a symmetric and reflexive relation on $X$ such that distinct elements $x,y \in X$ with $x \star y$ satisfy $t(x) \neq t(y)$. It is called the \textit{incidence relation} of $\Gamma$.
    \end{enumerate}
The \textit{rank} of $\Gamma$ is the cardinality of the type set $I$.
A \textit{flag} in an incidence system $\Gamma$ over $I$ is a set of pairwise incident elements. The type of a flag $F$ is $t(F)$, that is the set of types of the elements of $F.$ A \textit{chamber} is a flag of type $I$. An incidence system $\Gamma$ is an \textit{incidence geometry} if all its maximal flags are chambers.

Let $F$ be a flag of $\Gamma$. An element $x\in X$ is \textit{incident} to $F$ if $x\star y$ for all $y\in F$. The \textit{residue} of $\Gamma$ with respect to $F$, denoted by $\Gamma_F$, is the incidence system formed by all the elements of $\Gamma$ incident to $F$ but not in $F$. The \textit{rank} of a residue is equal to rank$(\Gamma)$ - $|F|$. For an element $x\in X$, we denote by $\text{Res}_{\Gamma}(x)$ the set of elements of $\Gamma_{\{x\}}$.

A geometry $\Gamma$ is \textit{firm} (respectively, \textit{thick}) if every flag of type other than $I$ is contained in at least two (respectively, three) distinct chambers of $\Gamma$. It is called \textit{thin} if every flag of type $I \backslash \{i\}$ for some $i \in I$ is contained in exactly two chambers of $\Gamma$.

The \textit{incidence graph} of $\Gamma$ is a graph with vertex set $X$ and where two elements $x$ and $y$ are connected by an edge if and only if $x \star y$ and $x \neq y$. Whenever we talk about the distance between two elements $x$ and $y$ of a geometry $\Gamma$, we mean the distance in the incidence graph of $\Gamma$ and simply denote it by $d_\Gamma(x,y)$, or even $d(x,y)$ if the context allows.
The geometry $\Gamma$ is \textit{residually connected} when the incidence graphs of all of its residues of rank at least $2$ are connected.

Let $\Gamma = \Gamma(X,\star,t)$ be an incidence system over the type set $I$. A \textit{correlation} is a bijection $\phi \colon X \to X$ such that it preserves incidence and sends elements of the same type to elements of the same type. More precisely, for every $x,y \in X$, $x \star y$ if and only if $\phi(x) \star \phi(y)$ and if $t(x) = t(y)$ then $t(\phi(x)) = t(\phi(y))$. If, moreover, $\phi$ fixes the types of every element (i.e $t(\phi(x)) = t(x)$ for all $x \in X$), then $\phi$ is said to be an \textit{automorphism}. The \textit{type} of a correlation $\phi$ is the permutation it induces on the type set $I$. A correlation of type $(i,j)$ is called a \textit{duality}. Additionally, if $\phi$ has order $2$, it is then called a \textit{polarity}. A correlation of type $(i,j,k)$ is called a \textit{triality}. The group of all correlations is denoted by $\Cor(\Gamma)$ and the group of automorphisms is denoted by $\Aut(\Gamma)$. Observe that $\Aut(\Gamma)$ is a normal subgroup of $\Cor(\Gamma)$ since it is the kernel of the action of $\Cor(\Gamma)$ on $I$.

A geometry $\Gamma$ such that the action of $\Aut(\Gamma)$ on the set of chambers is transitive is called \textit{flag-transitive} or \textit{chamber-transitive}.
Notice that being transitive on the chambers implies being transitive on the set of flags of any given type since, in a geometry, every flag can be extended to a chamber.
If moreover, the stabilizer of a chamber in $\Aut(\Gamma)$ is reduced to the identity we say that $\Gamma$ is \textit{simply transitive} or \textit{regular}.

Let $G$ be a group and $(G_i)_{i \in I}$ a collection of subgroups of $G$. The \textit{(left) coset geometry} $\Gamma = \Gamma(G, (G_i)_{i \in I})$ is the geometry over the type set $I$ where : 

\begin{enumerate}
    \item The elements of type $i$ are the left cosets of the form $xG_i, x \in G$.
    \item The element $xG_i$ is incident with $yG_j$ if and only if $i \neq j$ and $xG_i \cap yG_j \neq \emptyset$.
\end{enumerate}

The subgroups $(G_i)_{i \in I}$ are called \textit{maximal parabolics}.

For a given coset geometry, there is a natural incidence preserving left action of the group $G$ on the geometry $\Gamma$ by left multiplication.

$$G \times \Gamma \rightarrow \Gamma$$
$$(g,xG_i) \mapsto (gx)G_i$$

The \textit{Borel subgroup} is defined by $B = \bigcap \limits_{i \in I} G_i$. Notice that the action of $G$ on $\Gamma$ is faithful if and only if $B = \{e\}$.

\begin{proposition}\label{autog}
    Let $\Gamma(G,(G_i)_{i \in I})$ be a thin, residually connected coset geometry on which $G$ acts flag transitively and faithfully. Then $\Aut(\Gamma) \cong G$.
\end{proposition}

\begin{proof}
    Because of the thinness and residual connectedness, the stabilizer of a chamber is trivial in $\Aut(\Gamma)$. Hence, since $G$ acts flag transitively (and thus strictly chamber transitively), $G \cong \Aut(\Gamma)$.
\end{proof}

A thin and residually connected geometry is also called a \textit{hypertope}.

With the framework of incidence geometries in place, we next consider groups. In this second part of the section we introduce notations for several classical notions from group theory.

Let $G$ a group acting on a set $X$. 
When no ambiguity arises, we denote the identity element of a group $G$ by $e$. The order of an element $x \in G$ will be written $\text{ord}(x)$.

For a subset $A \subseteq X$, we denote by $$G_A = \{ h \in H \mid h(A) \subseteq A \}$$ the \textit{setwise stabilizer} (or \textit{stabilizer}) of $A$. 

If $H \leq G$, the \textit{$H$-orbit} of an element $x \in X$ under the action is the set $$x^H = \{ h . x \mid h \in H \}.$$

\section{Introduction to $\PR$-Graphs.}

We now introduce the main object of study of this paper. Let $\mathcal{G} = (V,E)$ be a finite simple graph admitting a proper 3-coloring of its edges $\phi \colon E \rightarrow \{ 1,2,3 \}$. Namely, if $x \in V$ and $e_1,e_2 \in E$ are distinct edges with $x \in e_1 \cap e_2$ then $\phi(e_1) \neq \phi(e_2)$. We call the triple $(V,E,\phi)$ a $\textit{$\PR$-Graph}$ for \textit{Permutation Representation Graph}.

Let $n = |V|$. Label the vertices of $V$ by $\{ v_1, \cdots , v_n \}$. We define three permutations $\rho_i$ for $1 \leq i \leq 3$ of $\Sym(V)$ the following way :

$$\rho_i(v_j) = 
    \begin{cases}
      v_k & \text{if } \{v_j,v_k\} \in E \text{ and } \phi(\{v_j,v_k\}) = i. \\
      v_j & \text{ otherwise.}
    \end{cases}
$$

Informally, each $\rho_i$ is a permutation that “swaps” the endpoints of every edge of color $i$ and leaves all other vertices fixed. Since every edge of color $i$ connects two distinct vertices, applying $\rho_i$ twice brings you back to the original labeling, so each $\rho_i$ is an involution. These three involutions capture the structure of the edge-coloring as a set of symmetries acting on the vertex set. The notion is analogous to that of \textit{C-group permutation representation graphs} (or \textit{CPR graphs}). However, in that setting, additional constraints ensure that the associated involutions generate a string $C$-group (see \cite{CPR1,CPR2}), whereas no such assumption is imposed in the present work. Throughout the remainder of the article, we identify each vertex $v_i$ with its label $i$. Thus, the vertex $v_i$ will simply be denoted by $i$ and we identify $\Sym(V)$ with the symmetric group $\Sv_n$ with the involutions $\rho_i$ acting on the indices.
We call these three involutions of $\Sv_n$ the induced involutions of the coloring $\phi$. We can associate to any $\PR$-Graph a well defined incidence system using coset geometry. Let
\begin{align}\label{cases}
    G = \langle \rho_1,\rho_2,\rho_3 \rangle, G_1 = \langle \rho_2, \rho_3 \rangle, G_2 = \langle \rho_1, \rho_3 \rangle, G_3 = \langle \rho_1, \rho_2 \rangle.
\end{align}

We define the rank $3$ incidence system $\Gamma(G,\{ G_1,G_2,G_3\})$. Note that the labeling of vertices is irrelevant, as any such choice yields an isomorphic group and therefore an isomorphic geometry. Many interesting properties of the geometry $\Gamma$ can be derived from the group $G$ and its maximal parabolics. In particular, since we are interested in residually connected geometries, we recall a group‑theoretic criterion that allows one to verify this property.

\begin{lemma}\label{RC}
    Let $\rho_1,\rho_2,\rho_3 \in \Sn$ be three distinct involutions and let $G,G_1,G_2,G_3$ be as in (\ref{cases}) and $\Gamma = \Gamma(G,\{ G_1,G_2,G_3\})$. Assume $G$ acts flag transitively on $\Gamma$. If $G_i \cap G_j = \langle \rho_k \rangle$ for all $i,j,k$ such that $\{ i,j,k \} = \{ 1,2,3 \}$ then the geometry $\Gamma$ is residually connected.
\end{lemma}

\begin{proof}
    This result can be found in \cite{DEHON94}. 
\end{proof}

The condition $G_i \cap G_j = \langle \rho_k \rangle$ is not generally satisfied, even for basic graph structures. Nonetheless, this property will turn out to be true for every tree $\PR$-Graph except some path graphs. We prove it in the next section.

\section{Generalities on $\PR$-Graphs that are trees.}

In this section $\mathcal{G} = (V,E,\phi)$ will always be a $\PR$-Graph that is also a tree (a connected simple graph with no cycles). We start by introducing the necessary vocabulary to discuss group actions on such graphs and prove some general results.

\begin{lemma}\label{faithful}
    Let $\mathcal{G} = (V,E,\phi)$ be a tree $\PR$-Graph, $x \in V$ and $B = G_1 \cap G_2 \cap G_3$. Then the $B$-orbit $x^B = \{ x \}$ and $B = \{ e \}$. Hence, the action of $G$ on $\Gamma$ is faithful.
\end{lemma}

\begin{proof}
Let $\tau \in G_1\cap G_2\cap G_3$ and $x \in V$.
We want to show that $\tau(x) = x$. Assume $\tau(x) \neq x$. Let $x = x_0, x_1,x_2,\cdots x_{m-1}, x_m = \tau(x)$ be the unique path (so $x_1 \neq x_0$) from $x$ to $\tau(x)$. Consider the color $j = \phi(\{x_0,x_1\})$. Since in particular $\tau \in G_j$ there exists another path from $x$ to $\tau(x)$ that does not use edges of color $j$ and thus does not use the edge $\{x_0,x_1\}$. Hence, there are two different paths from $x$ to $\tau(x)$ which is contradiction because $(V,E)$ has no cycles. Thus, $B = \{ e \}$ and the action of $G$ on $\Gamma$ is faithful.
\end{proof}

\begin{example}
    Notice the importance of $(V,E)$ being a tree, let $V$ be a triangle with vertices labeled $1,2,3$ and the involutions $\rho_1 = (1,2), \rho_2 = (2,3), \rho_3 = (3,1)$. Then $G_i = \Sv_3$ for all $1 \leq i \leq 3$ and $G_1 \cap G_2 \cap G_3 =  \Sv_3 \neq \{ e \}$.
\end{example}

The next lemma provides important observations about the action of a permutation in $G_i$ on a $G_i$-orbit. It will play a central role in several arguments throughout the paper. Let $\tau \in G_i \leq G$ and $\mathcal{O} \subseteq X$ be a $G_i$-orbit. We denote by $\tau \vert_\mathcal{O}$ the restriction of $\tau$ to $\mathcal{O}$, seen as a function $\tau \vert_\mathcal{O} \colon \mathcal{O} \rightarrow \mathcal{O}$ in $\Sym(\mathcal{O})$. 

\begin{lemma}\label{stab2}
    Let $\mathcal{G} = (V,E,\phi)$ be a tree $\PR$-Graph. For $\{i,j,k \} = \{ 1,2,3 \}$, let $\mathcal{O}$ be a $G_i$-orbit. Consider the homomorphism $\gamma \colon G_i \rightarrow \Sym(\mathcal{O})$ defined by $\gamma(\tau) = \tau \vert_{\mathcal{O}}$ and let $\pi \colon G_i / \text{Ker}(\gamma) \rightarrow \Sym(\mathcal{O})$ be defined by $\pi (\tau \text{Ker}(\gamma)) = \gamma(\tau)$. The following statements hold. 
    \begin{enumerate}
        \item $G_i / \text{Ker}(\gamma)$ acts on $\mathcal{O}$ as the natural action of the full automorphism group of an $| \mathcal{O} |$-gon.
        \item Whenever $| \mathcal{O} | > 2$, $\text{Ker}(\gamma) = \langle (\rho_j\rho_k)^{| \mathcal{O} |} \rangle$. 
        \item Whenever $| \mathcal{O} | = 2$, $\mathcal{O} = \{ a,b\}$ for some $a,b \in V$ such that $\{ a,b\} \in E$. Assume $\phi(\{ a,b \}) = k$. Then $\text{Ker}(\gamma)) = \langle \rho_j, (\rho_j\rho_k)^{2} \rangle$.
        \item Let $x \in \mathcal{O}$ and assume $| \mathcal{O} | > 1$. The set $\{ \tau \vert_{\mathcal{O}} \text{ with } \tau \in G_i \mid \tau(x) = x \}$ is of size $2$.
    \end{enumerate}
\end{lemma}

\begin{proof}
\begin{enumerate}
    \item Since $\mathcal{O}$ is a $G_i$-orbit, it forms a path in $V$ whose edges alternate between the two colors different from $i$. Assume $m = | \mathcal{O} |$ is even (the odd case is similar) and $\mathcal{O} = \{ 1, \cdots,m\}$ with the coloring $\phi(\{ 2l+1,2l+2\}) = 3$ and $\phi(\{ 2l+2,2l+3\}) = 2$ for $1 \leq l \leq \frac{| \mathcal{O} |-2}{2}$ as in Figure \ref{fig:meven1}. We can identify $\mathcal{O}$ with an $| \mathcal{O} |$-gon in the following way. Construct a regular $| \mathcal{O} |$-gon and label its vertices clockwise starting from the top by $1,2,4,\cdots,m,m-1,m-3,\cdots,3$. Under this identification, the action of $\rho_2 \text{Ker}(\gamma)$ on that $| \mathcal{O} |$-gon is a reflection alongside the vertical axis and the action of $\rho_3 \text{Ker}(\gamma)$ is a reflection alongside the vertical axis rotated clockwise by an angle $\frac{\pi}{m}$. Figure \ref{fig:ogon1} pictures the case $m = 8$.

\item It is clear that $\langle (\rho_j\rho_k)^{|\mathcal{O}|}\rangle \subseteq \text{Ker}(\gamma)$. Since $G_i = \langle \rho_j,\rho_k\rangle$ is dihedral, it can be written as $G_i = \rho_j \langle \rho_j\rho_k \rangle \sqcup \langle \rho_j\rho_k\rangle$. If $\rho_j (\rho_j\rho_k)^l \in \text{Ker}(\gamma)$, then  $\rho_j (\rho_j\rho_k)^r \in \text{Ker}(\gamma)$ where $r$ is the remainder of the division of $l$ by $| \mathcal{O} |$ and $\rho_j \notin \text{Ker}(\gamma)$ because $| \mathcal{O} | > 2$. It is impossible that $\rho_j (\rho_j\rho_k)^r \in \text{Ker}(\gamma)$ because it is a non trivial reflection of the $\mathcal{O}$-gon. If $(\rho_j\rho_k)^l \in \text{Ker}(\gamma)$, $(\rho_j\rho_k)^r \in \text{Ker}(\gamma)$ where $r$ is the remainder of the division of $l$ by $\mathcal{O}$ which has to be a multiple of $| \mathcal{O} |$ in order to vanish in the quotient. Hence $\langle (\rho_j\rho_k)^{|\mathcal{O}|}\rangle = \text{Ker}(\gamma)$.

\item Whenever $| \mathcal{O} | = 2$ and the only edge in that orbit is of color $k$, $\rho_j$ is the identity in the quotient $G_i/\text{Ker}(\gamma)$ but $\rho_k \notin \text{Ker}(\gamma)$. Hence $\text{Ker}(\gamma)$ is a proper subgroup of $G_i$ and since $\langle \rho_j, (\rho_j\rho_k)^{2} \rangle \subseteq \text{Ker}(\gamma)$, the conclusion follows.
\item Using part (1), we see that the stabilizer of a vertex of an $| \mathcal{O} |$-gon is always of size $2$.
\end{enumerate}

\end{proof}

\begin{figure}
\begin{center}
\begin{tikzpicture}[scale=1.5]

\def\r{2}      

\path [-] (0,-2.2) edge[green, double, very thick] (0,2.2);
\path [-] (67.5:2.2) edge[color=blue,dashed,very thick](-112.5:2.2);  

\node[style = thick,xshift=-8pt,yshift=3pt] (A) at (90:\r) {$1$};
\node[style = thick] (B)at ( 45:\r) {$2$};
\node[style = thick] (C) at (  0:\r) {$4$};
\node[style = thick] (D) at (-45:\r) {$6$};
\node[style = thick,xshift=-8pt,yshift=-3pt] (E) at (-90:\r) {$8$};
\node[style = thick] (F) at (-135:\r) {$7$};
\node[style = thick] (G) at (180:\r) {$5$};
\node[style = thick] (H) at (135:\r) {$3$};

\node[style = thick] (e) at (-60:1.2) {$\rho_2 \text{Ker}(\gamma)$};
\node[style = thick] (e) at (55:2.4) {$\rho_3 \text{Ker}(\gamma)$};

\draw[thick]
    (88.5:1.95)--(B)--(C)--(D)--(-88.5:1.95);
    
\draw[thick]
    (-91.5:1.95)--(F)--(G)--(H)--(91.5:1.95);

\end{tikzpicture}
\end{center}
\caption{$G_i / \text{Ker}(\gamma)$ acts on $\mathcal{O}$ as the natural action of the full automorphism group of an $| \mathcal{O} |$-gon.}
\label{fig:ogon1}
\end{figure}
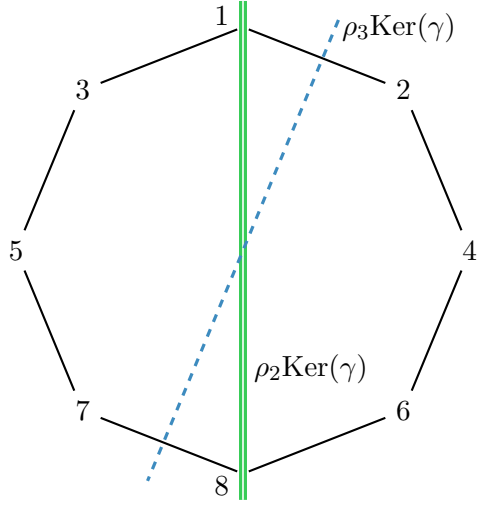

\begin{lemma}\label{4action}
    Let $\mathcal{G} = (V,E,\phi)$ be a tree $\PR$-Graph. For $1 \leq i \leq 3$, let $\mathcal{O}_1$ be a $G_i$-orbit with $m = |\mathcal{O}_1| > 2$ even and $\mathcal{O}_2$ a $G_i$-orbit of size $2$. Then $\tau \vert_{\mathcal{O}_1}$ fully determines $\tau \vert_{\mathcal{O}_2}$.
\end{lemma}

\begin{proof}
    Without loss of generality we can assume that $i = 1$. The pointwise stabilizer ($\text{Ker}(\gamma)$ in Lemma \ref{stab2}) of $\mathcal{O}_1$ in $G_1$ is $\langle (\rho_2 \rho_3)^{m} \rangle$. Let $\tau,\tau' \in G_1$ be two permutations such that $\tau \vert_{\mathcal{O}_1} = \tau' \vert_{\mathcal{O}_1}$. Then $\tau' \in \tau \langle (\rho_2 \rho_3)^{m} \rangle$. Since $m$ is even, $\tau$ and $\tau'$ are composed of the same parity (even or odd) of respectively $\rho_2$ and $\rho_3$ which fully determines $\tau \vert_{\mathcal{O}_2} = \tau' \vert_{\mathcal{O}_2}$. Note that if we had $| \mathcal{O}_1 | = 2$, the pointwise stabilizer of $\mathcal{O}_1$ would be $\langle \rho_2, (\rho_2 \rho_3)^2 \rangle$ or $\langle \rho_3, (\rho_2 \rho_3)^2 \rangle$ (see Lemma \ref{stab2}) and the argument would not work.
\end{proof}

\begin{theorem}\label{allthin}
    Let $\mathcal{G} = (V,E,\phi)$ be a tree $\PR$-Graph. $\mathcal{G}$ satisfies $G_i \cap G_j = \langle \rho_k \rangle$ for all $\{ i,j,k \} = \{1,2,3\}$ except if $\mathcal{G}$ is a path graph with vertices $\{1,\cdots,n\}$ and edges $\{i,i+1\}$ for $1 \leq i \leq n-1$ with $n = 2k_1 + 2k_2 + 1$ for $k_1, k_2 \geq 1$ with coloring
    $$\phi(\{ i,i+1 \}) = 
    \begin{cases}
      1 & \text{if } i \text{ is even and } 2 \leq i \leq 2k_1. \\
      2 & \text{if } i \text{ is odd and } 2k_1 + 1\leq i \leq 2k_1+2k_2-1. \\
      3 & \text{if } i \text{ is odd and } 1 \leq i \leq 2k_1-1. \\
      3 & \text{if } i \text{ is even and } 2k_1+2 \leq i \leq 2k_1+2k_2.
    \end{cases}
    $$
    as in Figure \ref{fig:nointersection} or one of its symmetries under the action of permuting the colors.
\end{theorem}

\begin{proof}
    The proof will be composed of two parts. Lemma \ref{generalthin} enforces a graph not satisfying that property to be a path graph. We then analyze the path graph case in detail in Lemma \ref{generalthin2}.
\end{proof}

\begin{lemma}\label{generalthin}
    Let $\mathcal{G} = (V,E,\phi)$ be a tree $\PR$-Graph with a vertex of degree $3$. Then $G_i \cap G_j = \langle \rho_k \rangle$ for $\{ 1,2,3 \} = \{ i,j,k\}$.
\end{lemma}
\begin{proof}
    Let $\tau \in G_1 \cap G_2$. Because $(V,E)$ is a tree, it has no cycles, thus for any $x \in V$, $\tau(x) = \rho_3(x)$ or $\tau(x) = x$. Let $\mathcal{O}_1$ and $\mathcal{O}_2$ be two $G_1$-orbits of size greater than $2$. Assume $\tau \vert_{\mathcal{O}_1} = \rho_3 \vert_{\mathcal{O}_1}$ and $\tau \vert_{\mathcal{O}_2} = e \vert_{\mathcal{O}_2}$. Since the pointwise stabilizer in $G_1$ of $\mathcal{O}_1$ is $\langle (\rho_2\rho_3)^{|\mathcal{O}_1|} \rangle$ and likewise the pointwise stabilizer in $G_1$ of $\mathcal{O}_2$ is $\langle (\rho_2\rho_3)^{|\mathcal{O}_2|} \rangle$ (see Lemma \ref{stab2}), we get $\tau \in \rho_3 \langle (\rho_2\rho_3)^{|\mathcal{O}_1|} \rangle$ and $\tau \in e \langle (\rho_2\rho_3)^{|\mathcal{O}_2|} \rangle$. Thus $\tau = \rho_3 (\rho_2\rho_3)^{k |\mathcal{O}_1|} = (\rho_2\rho_3)^{l |\mathcal{O}_2|}$ and hence $\rho_3 (\rho_2\rho_3)^{(k |\mathcal{O}_1|-l|\mathcal{O}_2|)} = e$ which is a contradiction. 
    Hence, the action of $\tau \in G_1 \cap G_2$ on every $G_1$-orbit of size greater than $2$ has to be either $\rho_3$ or $e$. We can then divide the graph in two parts : the $G_1$-orbits of size $2$ and the $G_1$-orbits of size greater than $2$. On each of these two parts, $\tau$ can only act as $\rho_3$ or $e$. Notice that if $\tau$ is known on one of the $G_1$-orbits of even size greater than $2$, it will automatically force the action of $\tau$ on the $G_1$-orbits of size $2$. If all the orbits of size greater than $2$ are odd then $(\rho_2\rho_3)^{\frac{n}{2}} \in G_1$ is a permutation that fixes all the $G_1$-orbits of size greater than $2$ and acts as $\rho_3$ in the orbits of size $2$ where $n = \text{ord}(\rho_2\rho_3)$ (and this is the only possibility). 
    Hence, $G_1 \cap G_2$ is a subgroup of a Klein four-group : $G_1 \cap G_2 \leq \{e,\rho_3,(\rho_2\rho_3)^{\frac{n}{2}},\rho_3(\rho_2\rho_3)^{\frac{n}{2}}\}$.
    The exact same argument can be done with $G_2$ instead of $G_1$. $G_1 \cap G_2 \leq \{e,\rho_3,(\rho_1\rho_3)^{\frac{m}{2}},\rho_3(\rho_1\rho_3)^{\frac{m}{2}}\}$ with $m = \text{ord}(\rho_1\rho_3)$.
    
    Hence, the only way to have the property $G_1 \cap G_2 = \langle \rho_3 \rangle$ to fail would be to have either $(\rho_2\rho_3)^{\frac{n}{2}} = (\rho_1\rho_3)^{\frac{m}{2}}$ or $(\rho_2\rho_3)^{\frac{n}{2}} = \rho_3(\rho_1\rho_3)^{\frac{m}{2}}$ with $m,n$ divisible by $2$ but not divisible by $4$ meaning for $1 \leq i \leq 2$ there exists at least one $G_i$-orbits of size greater than $2$, all such orbits being of odd size and there exists at least one $G_i$-orbits of size $2$ where the edge in that orbit is of color $3$.
    
    We first treat the case $\tau = (\rho_2\rho_3)^{\frac{n}{2}} = (\rho_1\rho_3)^{\frac{m}{2}}$. The permutation $(\rho_2\rho_3)^{\frac{n}{2}}$ exchanges the endpoints of the $G_1$-orbits of size $2$ (with an edge of color $3$) and leaves the rest unchanged while $(\rho_1\rho_3)^{\frac{m}{2}}$ exchanges the endpoints of the $G_2$-orbits of size two (with an edge of color $3$) and leaves the rest unchanged. Thus a $G_1$-orbit of size $2$ (with an edge of color $3$) would also be a $G_2$-orbit of size $2$ (with an edge of color $3$) which is impossible since the graph is connected (see Figure \ref{fig:interprop1}).
    
    We now treat the second case where $\tau = (\rho_2\rho_3)^{\frac{n}{2}} = \rho_3(\rho_1\rho_3)^{\frac{m}{2}}$. The element $(\rho_2\rho_3)^{\frac{n}{2}}$ exchanges the endpoints of the $G_1$-orbits of size $2$ of color $3$ and leaves the rest unchanged while $\rho_3(\rho_1\rho_3)^{\frac{m}{2}}$ exchanges the endpoints of the edges of color $3$ in every $G_2$-orbit of size greater than $2$ and leaves the rest unchanged. Hence, every $G_1$-orbit of size greater than $2$ must be unchanged. Let $x$ be a vertex of degree $3$ surrounded by three neighbors labeled $1,2,3$ as in Figure \ref{fig:interprop2}. Then $\{ 3,x,2 \}$ is a subset of a $G_1$-orbit of size greater than two. So $\tau(x) = x$ and $\tau(x) = 2$. Since $\{ 1,x,2 \}$ is a subset of a $G_2$-orbit of size greater than three, $\tau(x) = 2$ and $\tau(2) = x$ which is a contradiction. Thus $G_1 \cap G_2 = \langle \rho_3 \rangle$. The same argument can be used to prove that $G_i \cap G_j = \langle \rho_k \rangle$ for $\{ 1,2,3 \} = \{ i,j,k\}$ which concludes the proof.

    Hence, the only trees where $G_1 \cap G_2 = \langle \rho_3 \rangle$ fails are path graphs and the only way it can fail is if $(\rho_2\rho_3)^{\frac{n}{2}} = \rho_3(\rho_1\rho_3)^{\frac{m}{2}}$ with $n = \text{ord}(\rho_2 \rho_3), m = \text{ord}(\rho_1 \rho_3)$ and there are no $G_1$-orbits or $G_2$-orbits of size greater than $2$ and even.
    
    \begin{figure}
    \begin{center}
    
    \begin{tikzpicture}[scale = 1.5]
    
    \node[style = thick] (2) at (0,0) {$2$};
    \node[style = thick] (3) at (1,0) {$3$};
    \node[style = thick] (4) at (1+1/2,0.86602540378) {$4$};
    
    \node[style = thick] (1) at (-1/2,0.86602540378) {$1$};
    \node[style = thick] (5) at (-1/2,-0.86602540378) {$5$};
    \node[style = thick] (6) at (1+1/2,-0.86602540378) {$6$};
    
    \path [-] (2) edge[color=blue,dashed,very thick] (3);
    \path [-] (1) edge[color=red,very thick] (2);
    \path [-] (2) edge[color=green,double,very thick] (5);
    \path [-] (3) edge[color=red,very thick] (4);
    \path [-] (3) edge[color=green,double,very thick] (6);
    
\end{tikzpicture}
\end{center}
\caption{A $G_1$-orbit of size $2$ of color $3$ cannot be a $G_2$-orbit of size $2$ of color $3$.}
\label{fig:interprop1}
\end{figure}
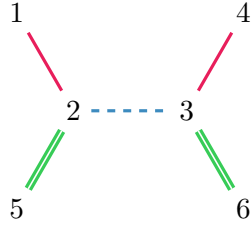

    \begin{figure}
    \begin{center}
    
    \begin{tikzpicture}[scale = 1.5]
    
    \node[style = thick] (x) at (0,0) {$x$};
    \node[style = thick] (2) at (1,0) {$2$};
    \node[style = thick] (1) at (-1/2,0.86602540378) {$1$};
    \node[style = thick] (3) at (-1/2,-0.86602540378) {$3$};
    
    \path [-] (x) edge[color=green,double,very thick] (3);
    \path [-] (x) edge[color=red,very thick] (1);
    \path [-] (x) edge[color=blue,dashed,very thick] (2);
    
\end{tikzpicture}
\end{center}
\caption{A vertex of degree $3$.}
\label{fig:interprop2}
\end{figure}
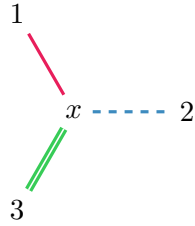
    
\end{proof}

\begin{lemma}\label{generalthin2}
    Let $\mathcal{G} = (V,E,\phi)$ be a path $\PR$-Graph such that $G_i \cap G_j \neq \langle \rho_k \rangle$ for certain $\{ 1,2,3 \} = \{ i,j,k\}$. Then $\mathcal{G} = (V,E,\phi)$ is a path graph $\{1,\cdots,n\}$ with vertices $\{1,\cdots,n\}$ and edges $\{i,i+1\}$ for $1 \leq i \leq n-1$ with $n = 2k_1 + 2k_2 + 1$ for $k_1, k_2 \geq 1$ with coloring
    $$\phi(\{ i,i+1 \}) = 
    \begin{cases}
      1 & \text{if } i \text{ is even and } 2 \leq i \leq 2k_1. \\
      2 & \text{if } i \text{ is odd and } 2k_1 + 1\leq i \leq 2k_1+2k_2-1. \\
      3 & \text{if } i \text{ is odd and } 1 \leq i \leq 2k_1-1. \\
      3 & \text{if } i \text{ is even and } 2k_1+2 \leq i \leq 2k_1+2k_2.
    \end{cases}
    $$
    as in Figure \ref{fig:nointersection} or one of its symmetries under the action of permuting the colors.
\end{lemma}

\begin{proof}
    Assume $G_1 \cap G_2 \neq \langle \rho_3 \rangle$. We start with some observations to exclude certain configurations. As in the previous proof, all the edges of color $3$ belonging to a $G_1$-orbit of size greater than $2$ need to be an edge of color $3$ of a $G_2$-orbit of size $2$. Assume a subgraph of $(V,E)$ is given by $V' = \{ a,b,c,d \} \subseteq V$ and $E' = \{ \{ a,b\} , \{ b,c\} , \{ c,d\} \} \subseteq E$ with the induced coloring from $\phi$ such that $\phi(\{a,b\}) = 1,\phi(\{b,c\}) = 3,\phi(\{c,d\}) = 2$. Then $\{ a,b,c \}$ is a subset of a $G_2$-orbit and $\{b,c,d\}$ is a subset of a $G_1$-orbit both containing the edge $\{b,c\}$ of color $3$ which is impossible. This situation is represented in the top row of Figure \ref{fig:forbiddenpatterns}. The bottom configurations of the same figure are also forbidden since they include a $G_1$-orbit (or $G_2$-orbit) of size $2$ which is not of color $3$. Indeed, $\tau = (\rho_2\rho_3)^{\frac{n}{2}} = \rho_3(\rho_2\rho_3)^{\frac{m}{2}} \in G_1 \cap G_2$ would satisfy $\tau(b) = c, \tau(c)=b$ but would not respect the orbits of $G_1 \cap G_2$.
    
    Start with a $G_2$-orbit of odd size $2k_1+1$ as in figure \ref{fig:nointersection}. There cannot be an edge of color $1$ on its left otherwise it would create the top right pattern of Figure \ref{fig:forbiddenpatterns}. Thus, it can only be colored on the right. Place an edge of color $2$ on the right. The edge after that one cannot be of color $1$ since it would create the pattern on the bottom right of Figure \ref{fig:forbiddenpatterns}, so it is of color $3$. The one after that cannot be of color $1$ otherwise it would create the pattern on the top right. So it is of color $2$. We realize that we can only alternate between color $2$ and $3$ and need to stop on color $3$ to ensure the $G_2$-orbit is of odd size $2k_2+1$. 
    
    This forms the graph in Figure \ref{fig:nointersection} with $2k_1+1$ the size of the $G_1$-orbit of size greater than $2$ and $2k_2+1$ the size of the $G_2$-orbit of size greater than $2$. It is easy to check that these colored path graphs do not verify $G_1 \cap G_2 = \langle \rho_3 \rangle$ since $(\rho_1\rho_3)^{2k_1+1} = \rho_3 (\rho_2 \rho_3)^{2k_2+1} \in G_1 \cap G_2$.
\end{proof}

\begin{figure}
    \begin{center}
    
\begin{tikzpicture}[scale = 1.5]
    
    \draw [red, very thick] (0,0) -- (1,0) ;
    \draw [blue, dashed, very thick] (1,0) -- (2,0) ;
    \draw [green, double, very thick] (2,0) -- (3,0) ;
    
    \draw [green, double, very thick] (5,0) -- (6,0) ;
    \draw [blue, dashed, very thick] (6,0) -- (7,0) ;
    \draw [red, very thick] (7,0) -- (8,0) ;
    
    \draw [green, double, very thick] (0,-1.5) -- (1,-1.5) ;
    \draw [red, very thick] (1,-1.5) -- (2,-1.5) ;
    \draw [green, double, very thick] (2,-1.5) -- (3,-1.5) ;

    \draw [red, very thick] (5,-1.5) -- (6,-1.5) ;
    \draw [green, double, very thick] (6,-1.5) -- (7,-1.5) ;
    \draw [red, very thick] (7,-1.5) -- (8,-1.5) ;
    
    \draw (0,0) node[above]{$a$};`
    \draw (1,0) node[above]{$b$};`
    \draw (2,0) node[above]{$c$};`
    \draw (3,0) node[above]{$d$};`

    \draw (5,0) node[above]{$a$};`
    \draw (6,0) node[above]{$b$};`
    \draw (7,0) node[above]{$c$};`
    \draw (8,0) node[above]{$d$};`

    \draw (0,-1.5) node[above]{$a$};`
    \draw (1,-1.5) node[above]{$b$};`
    \draw (2,-1.5) node[above]{$c$};`
    \draw (3,-1.5) node[above]{$d$};`

    \draw (5,-1.5) node[above]{$a$};`
    \draw (6,-1.5) node[above]{$b$};`
    \draw (7,-1.5) node[above]{$c$};`
    \draw (8,-1.5) node[above]{$d$};`
    
    \node at (0,0)[circle,fill,inner sep=1.5pt,color=black]{};
    \node at (1,0)[circle,fill,inner sep=1.5pt,color=black]{};
    \node at (2,0)[circle,fill,inner sep=1.5pt,color=black]{};
    \node at (3,0)[circle,fill,inner sep=1.5pt,color=black]{};
    
    \node at (0,-1.5)[circle,fill,inner sep=1.5pt,color=black]{};
    \node at (1,-1.5)[circle,fill,inner sep=1.5pt,color=black]{};
    \node at (2,-1.5)[circle,fill,inner sep=1.5pt,color=black]{};
    \node at (3,-1.5)[circle,fill,inner sep=1.5pt,color=black]{};

    \node at (5,0)[circle,fill,inner sep=1.5pt,color=black]{};
    \node at (6,0)[circle,fill,inner sep=1.5pt,color=black]{};
    \node at (7,0)[circle,fill,inner sep=1.5pt,color=black]{};
    \node at (8,0)[circle,fill,inner sep=1.5pt,color=black]{};
    
    \node at (5,-1.5)[circle,fill,inner sep=1.5pt,color=black]{};
    \node at (6,-1.5)[circle,fill,inner sep=1.5pt,color=black]{};
    \node at (7,-1.5)[circle,fill,inner sep=1.5pt,color=black]{};
    \node at (8,-1.5)[circle,fill,inner sep=1.5pt,color=black]{};

\end{tikzpicture}
\end{center}
\caption{The forbidden patterns.}
\label{fig:forbiddenpatterns}
\end{figure}
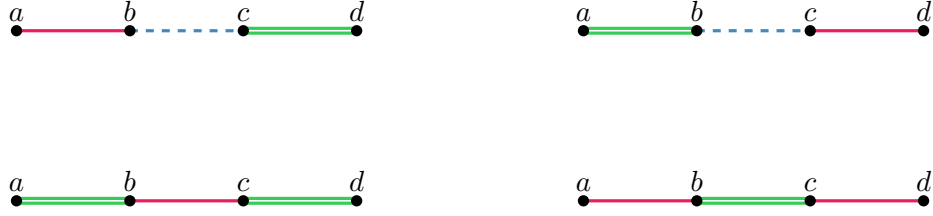

\begin{figure}
    \begin{center}
    
\begin{tikzpicture}[scale = 1.3]
    
    \draw [blue, dashed, dashed, very thick] (0,0) -- (1,0) ;
    \draw [red, very thick] (1,0) -- (2,0) ;
    \draw [blue, dashed, dashed, very thick] (2,0) -- (2.2,0) ;
    \node at (2.4,0)[circle,fill,inner sep=0.5pt,color=black]{};
    \node at (2.5,0)[circle,fill,inner sep=0.5pt,color=black]{};
    \node at (2.6,0)[circle,fill,inner sep=0.5pt,color=black]{};
    \draw [red, very thick] (2.8,0) -- (3,0) ;
    \draw [blue, dashed, dashed, very thick] (3,0) -- (4,0) ;
    \draw [red, very thick] (4,0) -- (5,0) ;
    \draw [green, double, very thick] (5,0) -- (6,0) ;
    \draw [blue, dashed, dashed, very thick] (6,0) -- (7,0) ;
    \draw [green, double, very thick] (7,0) -- (7.2,0) ;
    \node at (7.4,0)[circle,fill,inner sep=0.5pt,color=black]{};
    \node at (7.5,0)[circle,fill,inner sep=0.5pt,color=black]{};
    \node at (7.6,0)[circle,fill,inner sep=0.5pt,color=black]{};
    \draw [blue, dashed, dashed, very thick] (7.8,0) -- (8,0) ;
    \draw [green, double, very thick] (8,0) -- (9,0) ;
    \draw [blue, dashed, dashed, very thick] (9,0) -- (10,0) ;
    
    \draw (0,0) node[above]{$1$};`
    \draw (1,0) node[below]{$2$};`
    \draw (2,0) node[above]{$3$};
    \draw (3,0) node[below]{$2k_1-1$};``
    \draw (4,0) node[above]{$2k_1$};`
    \draw (5,0) node[below]{$2k_1+1$};
    \draw (6,0) node[above]{$2k_1+2$};``
    \draw (7,0) node[below]{$2k_1+3$};`
    \draw (8,0) node[above]{$2k_1+2k_2-1$};
    \draw (9,0) node[below]{$2k_1+2k_2$};`
    \draw (10,0) node[above]{$2k_1+2k_2+1$};
    
    \node at (0,0)[circle,fill,inner sep=1.5pt,color=black]{};
    
    \node at (1,0)[circle,fill,inner sep=1.5pt,color=black]{};
    
    \node at (2,0)[circle,fill,inner sep=1.5pt,color=black]{};
    \node at (3,0)[circle,fill,inner sep=1.5pt,color=black]{};
    
    \node at (4,0)[circle,fill,inner sep=1.5pt,color=black]{};
    
    \node at (5,0)[circle,fill,inner sep=1.5pt,color=black]{};
    
    \node at (6,0)[circle,fill,inner sep=1.5pt,color=black]{};
    \node at (7,0)[circle,fill,inner sep=1.5pt,color=black]{};
    \node at (8,0)[circle,fill,inner sep=1.5pt,color=black]{};
    \node at (9,0)[circle,fill,inner sep=1.5pt,color=black]{};
    \node at (10,0)[circle,fill,inner sep=1.5pt,color=black]{};

\end{tikzpicture}
\end{center}
\caption{The only tree $\PR$-Graphs with $G_1 \cap G_2 \neq \langle \rho_3 \rangle$, $k_1,k_2 \geq 1$.}
\label{fig:nointersection}
\end{figure}
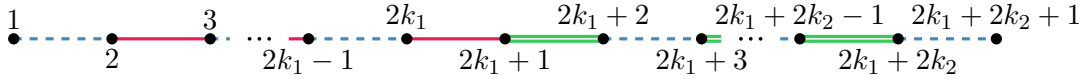

The following lemma provides a useful tool for testing the flag transitivity of the geometry $\Gamma$.

\begin{lemma}\label{FT}
    If $G_1,G_2,G_3 \leq G$ satisfy $(G_1 \cap G_2)(G_1 \cap G_3) = G_1 \cap G_2G_3$. Then $\Gamma$ is flag transitive.
\end{lemma} 

\begin{proof}
     The earliest appearance of this result is found in \cite{Tits1974}. When the groups $G_1,G_2,G_3$ satisfy $(G_1 \cap G_2)(G_1 \cap G_3) = G_1 \cap G_2G_3$ then Proposition 4.2 of \cite{hypertopes} holds. Consequently, Theorem 3.5 of the same work follows, ensuring that the geometry is indeed flag transitive.
\end{proof}

Notice that if $G_i \cap G_j = \langle \rho_k \rangle$ for $\{ i,j,k \} = \{ 1,2,3 \}$ then $\{ e , \rho_3, \rho_2, \rho_3\rho_2 \} = (G_1 \cap G_2)(G_1 \cap G_3)$. The inclusion $\{ e , \rho_3, \rho_2, \rho_3\rho_2 \} \subseteq G_1 \cap G_2G_3$ is always true. 

Let $\mathcal{O}$ a $G_1$-orbit. If $\tau \in G_1 \cap G_2G_3$, it imposes strong conditions on what $\tau\vert_{\mathcal{O}}$ can be. The following lemma characterizes the possible restrictions $\tau\vert_{\mathcal{O}}$ and will be used extensively to prove flag transitivity of $\Gamma$ using Lemma \ref{FT}.

\begin{proposition}\label{actionorbit}
    Let $\tau \in G_1 \cap G_2G_3$ and $\mathcal{O}$ be a $G_1$-orbit. The restriction $\tau \vert_{\mathcal{O}} \colon \mathcal{O} \rightarrow \mathcal{O}$ is equal to one of the restrictions $e \vert_\mathcal{O} , \rho_3 \vert_{\mathcal{O}}, \rho_2 \vert_\mathcal{O}$ or $\rho_3\rho_2 \vert_\mathcal{O}$.
\end{proposition}

\begin{proof}
    For convenience assume $\mathcal{O} = \{ 1, \cdots , m\}$. Since $\tau \in G_1$, $\mathcal{O}$ is a path with alternating colors $2$ and $3$. There are essentially three distinct cases to analyze. 
    \begin{enumerate}
        \item The edges $\{ 1,2 \}$ and $\{ m-1, m \}$ are such that $\phi(\{ 1,2 \}) = \phi(\{ m-1, m \}) = 3$. In this case $m$ is even (see Figure \ref{fig:meven1}).
        \item The edges $\{ 1,2 \}$ and $\{ m-1, m \}$ are such that $\phi(\{ 1,2 \}) = \phi(\{ m-1, m \}) = 2$. In this case $m$ is even (see Figure \ref{fig:meven2}).
        \item The edges $\{ 1,2 \}$ and $\{ m-1, m \}$ are such that $\phi(\{ 1,2 \}) = 3 $ and $\phi(\{ m-1, m \}) = 2$. In this case $m$ is odd (see Figure \ref{fig:modd}).
    \end{enumerate}
    
    We start with the first case. If $\tau\in G_2G_3$, it is written as $\tau_2\tau_3$ where $\tau_2 \in G_2, \tau_3 \in G_3$.
    Hence for $\tau(1)$ to be in the orbit $1^{G_1} = \mathcal{O}$, we need $\tau_3$ to either fix 1 or send it to $\rho_1(1)$ so that $\tau_2(\tau_3(1))$ is in $\mathcal{O}$.
    Because $(V,E)$ is a tree, we get $\tau(1) \in \{ 1,2 \}$. By Lemma \ref{stab2}, if $\tau(1) = 1$ then $\tau \vert_\mathcal{O} \in \{ e \vert_{\mathcal{O}}, \rho_2 \vert_\mathcal{O}\}$. If $\tau(1) = 2$ then $\tau \vert_\mathcal{O} \in \{ \rho_3 \vert_{\mathcal{O}}, \rho_3\rho_2 \vert_\mathcal{O} \}$. In any case, $\tau \vert_O \in \{ e \vert_{\mathcal{O}}, \rho_2 \vert_\mathcal{O} , \rho_3 \vert_{\mathcal{O}}, \rho_3\rho_2 \vert_\mathcal{O} \}$.
    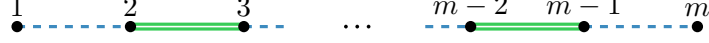
\begin{figure}
    \begin{center}
    
\begin{tikzpicture}[scale = 1.5]
    
    \draw [blue, dashed, dashed, very thick] (0,0) -- (1,0) ;
    \draw [green, double, very thick] (1,0) -- (2,0) ;
    \draw [blue, dashed, dashed, very thick] (2,0) -- (2.4,0) ;
    \node at (2.9,0)[circle,fill,inner sep=0.5pt,color=black]{};
    
    \node at (3,0)[circle,fill,inner sep=0.5pt,color=black]{};
    
    \node at (3.1,0)[circle,fill,inner sep=0.5pt,color=black]{};
    \draw [blue, dashed, dashed, very thick] (3.6,0) -- (4,0) ;
    \draw [green, double, very thick] (4,0) -- (5,0) ;
    \draw [blue, dashed, dashed, very thick] (5,0) -- (6,0) ;

    \draw (0,0) node[above]{$1$};`
    \draw (1,0) node[above]{$2$};`
    \draw (2,0) node[above]{$3$};
    \draw (4,0) node[above]{$m-2$};``
    \draw (5,0) node[above]{$m-1$};`
    \draw (6,0) node[above]{$m$};
    
    \node at (0,0)[circle,fill,inner sep=1.5pt,color=black]{};
    
    \node at (1,0)[circle,fill,inner sep=1.5pt,color=black]{};
    
    \node at (2,0)[circle,fill,inner sep=1.5pt,color=black]{};
    
    \node at (4,0)[circle,fill,inner sep=1.5pt,color=black]{};
    
    \node at (5,0)[circle,fill,inner sep=1.5pt,color=black]{};
    
    \node at (6,0)[circle,fill,inner sep=1.5pt,color=black]{};

\end{tikzpicture}
\end{center}
\caption{The first case with $m$ even.}
\label{fig:meven1}
\end{figure}

For the second case. If $m=2$ there is nothing to prove. Assume $m > 2$ even. Again because $(V,E)$ is a tree, and $\tau \in G_2G_3$, notice that $\tau(1) \in \{ 1,2,3 \}$. For every $x \in \{ 1,2,3 \}$ there are exactly two $\tau \vert_{\mathcal{O}}$ such that $\tau\vert_{\mathcal{O}}(1) = x$ (Lemma \ref{stab2}). We get $\tau \vert_{\mathcal{O}} \in \{ e\vert_{\mathcal{O}}, \rho_3 \vert_{\mathcal{O}} , \rho_2 \vert_{\mathcal{O}} , \rho_3 \rho_2 \vert_{\mathcal{O}} , \rho_2 \rho_3 \vert_{\mathcal{O}}, \rho_3 \rho_2 \rho_3 \vert_{\mathcal{O}} \}$. If $\tau \vert_{\mathcal{O}} =\rho_2 \rho_3 \vert_{\mathcal{O}}$ or $\tau \vert_{\mathcal{O}} =\rho_3 \rho_2 \rho_3 \vert_{\mathcal{O}}$ then $\tau(3) = 1$ which is a contradiction since $\tau \in G_2G_3$.
Hence, $\tau \vert_{\mathcal{O}} \in \{ e\vert_{\mathcal{O}}, \rho_3 \vert_{\mathcal{O}} , \rho_2 \vert_{\mathcal{O}} , \rho_3 \rho_2 \vert_{\mathcal{O}} \}$.
    
\begin{figure}
\begin{center}
\begin{tikzpicture}[scale = 1.5]
    
    \draw [green, double, very thick] (0,0) -- (1,0) ;
    \draw [blue, dashed, very thick] (1,0) -- (2,0) ;
    \draw [blue, dashed, very thick] (4,0) -- (5,0) ;
    \draw [green, double, very thick] (5,0) -- (6,0) ;
    \draw [green, double, very thick] (2,0) -- (2.4,0) ;
    \node at (2.9,0)[circle,fill,inner sep=0.5pt,color=black]{};
    
    \node at (3,0)[circle,fill,inner sep=0.5pt,color=black]{};
    
    \node at (3.1,0)[circle,fill,inner sep=0.5pt,color=black]{};
    \draw [green, double, very thick] (3.6,0) -- (4,0) ;
    \draw (0,0) node[above]{$1$};`
    \draw (1,0) node[above]{$2$};`
    \draw (2,0) node[above]{$3$};
    \draw (4,0) node[above]{$m-2$};``
    \draw (5,0) node[above]{$m-1$};`
    \draw (6,0) node[above]{$m$};
    
    \node at (0,0)[circle,fill,inner sep=1.5pt,color=black]{};
    
    \node at (1,0)[circle,fill,inner sep=1.5pt,color=black]{};
    
    \node at (2,0)[circle,fill,inner sep=1.5pt,color=black]{};
    
    \node at (4,0)[circle,fill,inner sep=1.5pt,color=black]{};
    
    \node at (5,0)[circle,fill,inner sep=1.5pt,color=black]{};
    
    \node at (6,0)[circle,fill,inner sep=1.5pt,color=black]{};

\end{tikzpicture}
\end{center}
\caption{The second case with $m$ even.}
\label{fig:meven2}
\end{figure}
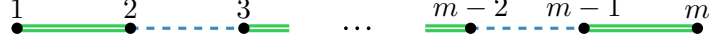
    We end with the third case. With the same reasoning as for the other cases, $\tau(1) \in \{1,2\}$ which forces $\tau \vert_{\mathcal{O}} \in \{ e\vert_{\mathcal{O}}, \rho_3 \vert_{\mathcal{O}} , \rho_2 \vert_{\mathcal{O}} , \rho_3 \rho_2 \vert_{\mathcal{O}} \}$.
    
\begin{figure}
\begin{center}
\begin{tikzpicture}[scale = 1.5]
    
    \draw [blue, dashed, very thick] (0,0) -- (1,0) ;
    \draw [green, double, very thick] (1,0) -- (2,0) ;
    \draw [blue, dashed, very thick] (4,0) -- (5,0) ;
    \draw [green, double, very thick] (5,0) -- (6,0) ;
    \draw [blue, dashed, dashed, very thick] (2,0) -- (2.4,0) ;
    \node at (2.9,0)[circle,fill,inner sep=0.5pt,color=black]{};
    
    \node at (3,0)[circle,fill,inner sep=0.5pt,color=black]{};
    
    \node at (3.1,0)[circle,fill,inner sep=0.5pt,color=black]{};
    \draw [green, double, very thick] (3.6,0) -- (4,0) ;
    
    \draw (0,0) node[above]{$1$};`
    \draw (1,0) node[above]{$2$};`
    \draw (2,0) node[above]{$3$};
    \draw (4,0) node[above]{$m-2$};``
    \draw (5,0) node[above]{$m-1$};`
    \draw (6,0) node[above]{$m$};
    
    \node at (0,0)[circle,fill,inner sep=1.5pt,color=black]{};
    
    \node at (1,0)[circle,fill,inner sep=1.5pt,color=black]{};
    
    \node at (2,0)[circle,fill,inner sep=1.5pt,color=black]{};
    
    \node at (4,0)[circle,fill,inner sep=1.5pt,color=black]{};
    
    \node at (5,0)[circle,fill,inner sep=1.5pt,color=black]{};
    
    \node at (6,0)[circle,fill,inner sep=1.5pt,color=black]{};

\end{tikzpicture}
\end{center}

\caption{The case with $m$ odd.}

\label{fig:modd}
\end{figure}
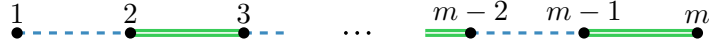

In every case we find that $\tau \vert_{\mathcal{O}} \in \{ e \vert_{\mathcal{O}} , \rho_3 \vert_{\mathcal{O}}, \rho_2 \vert_{\mathcal{O}} , \rho_3\rho_2 \vert_{\mathcal{O}} \}$ which concludes the proof.
\end{proof}

The previous proposition imposes strong conditions for the action of $\tau \in G_1 \cap G_2G_3$ on the $G_1$-orbits. The next example showcases a $\PR$-Graph that does not produce a flag transitive geometry.

\begin{example}
    Let $n = 10$ and consider the path $\PR$-Graph defined by $$\rho_1 = (1, 2)(4, 5)(6, 7)(9, 10),  \rho_2=(2, 3)(5, 6)(8, 9),\rho_3=(3, 4)(7, 8)$$ as in Figure \ref{fig:notft}.
    In that case $\Gamma$ is not flag transitive. Indeed, computation with {\sc Magma} shows that $G_1 \cap G_2G_3$ is composed of 8 elements. For example $(3, 4)(5, 6)(7, 8) = \rho_3 (\rho_2\rho_3)^3 = (\rho_1 \rho_3)^2(\rho_2\rho_1)^3 \in G_1 \cap G_2G_3$.
    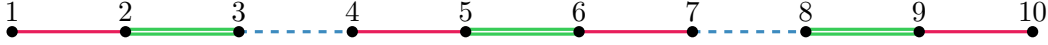
\begin{figure}
\begin{center}
\begin{tikzpicture}[scale = 1.5]
    
    \draw [red, very thick] (0,0) -- (1,0) ;
    \draw [green, double, very thick] (1,0) -- (2,0) ;
    \draw [blue, dashed, very thick] (2,0) -- (3,0) ;
    \draw [red, very thick] (3,0) -- (4,0) ;
    \draw [green, double, very thick] (4,0) -- (5,0) ;
    \draw [red, very thick] (5,0) -- (6,0) ;
    \draw [blue, dashed, very thick] (6,0) -- (7,0) ;
    \draw [green, double, very thick] (7,0) -- (8,0) ;
    \draw [red, very thick] (8,0) -- (9,0) ;
    
    \draw (0,0) node[above]{$1$};
    \draw (1,0) node[above]{$2$};
    \draw (2,0) node[above]{$3$};
    \draw (3,0) node[above]{$4$};
    \draw (4,0) node[above]{$5$};
    \draw (5,0) node[above]{$6$};
    \draw (6,0) node[above]{$7$};
    \draw (7,0) node[above]{$8$};
    \draw (8,0) node[above]{$9$};
    \draw (9,0) node[above]{$10$};
    
    \node at (0,0)[circle,fill,inner sep=1.5pt,color=black]{};
    
    \node at (1,0)[circle,fill,inner sep=1.5pt,color=black]{};
    
    \node at (2,0)[circle,fill,inner sep=1.5pt,color=black]{};
    
    \node at (3,0)[circle,fill,inner sep=1.5pt,color=black]{};
    \node at (4,0)[circle,fill,inner sep=1.5pt,color=black]{};
    
    \node at (5,0)[circle,fill,inner sep=1.5pt,color=black]{};
    
    \node at (6,0)[circle,fill,inner sep=1.5pt,color=black]{};
    \node at (7,0)[circle,fill,inner sep=1.5pt,color=black]{};
    \node at (8,0)[circle,fill,inner sep=1.5pt,color=black]{};
    \node at (9,0)[circle,fill,inner sep=1.5pt,color=black]{};

\end{tikzpicture}
\end{center}
\caption{A path $\PR$-Graph such that $\Gamma$ is not flag transitive.}
\label{fig:notft}
\end{figure}
Notice that it does not contradict proposition \ref{actionorbit}. In this case $G = \langle \rho_1, \rho_2 , \rho_3 \rangle = C_2 \wr \Sv_5$. The example is available on the repository.
\end{example}

With the basic framework in place, we now proceed to the following sections where we construct families of geometries $\Gamma$ based on $\PR$-graphs whose automorphism groups are either the symmetric or the alternating group.

\section{Family 1 : the trident graphs.}

In this section we analyse the first infinite families of tree $\PR$-Graphs. One key feature is that the size of the maximal parabolics is constant. We refer to these graphs as \textit{trident graphs}, they are depicted in Figure \ref{family1}.

\begin{figure}
    \centering
\begin{tikzpicture}
\begin{scope}[every node/.style={thick}]
    \node (1) at (0,0) {$1$};
    \node (2) at (1.5,-2) {$2$};
    \node (3) at (1.5,0) {$3$};
    \node (4) at (1.5,2) {$4$};
    \node (5) at (3,-2) {$5$};
    \node (6) at (3,0) {$6$} ;
    \node (7) at (3,2) {$7$} ;
    \node (8) at (4.5,-2) {$8$} ;
    \node (9) at (4.5,0) {$9$} ;
    \node (10) at (4.5,2) {$10$} ;
    \node (11) at (6,-2) {$11$} ;
    \node (12) at (6,0) {$12$} ;
    \node (13) at (6,2) {$13$} ;
    \node (14) at (7.5,-2) {$14$} ;
    \node (15) at (7.5,0) {$15$} ;
    \node (16) at (7.5,2) {$16$} ;
    \node (17) at (9,-2) {$17$} ;
    \node (18) at (9,0) {$18$} ;
    \node (19) at (9,2) {$19$} ;
    \node (20) at (10.5,-2) {} ;
    \node (21) at (10.5,0) {} ;
    \node (22) at (10.5,2) {} ;
    \node (23) at (12,-2) {$11+9k$} ;
    \node (24) at (12,0) {$12+9k$} ;
    \node (25) at (12,2) {$13+9k$} ;
    
\end{scope}

\begin{scope}[>={Stealth[black]},
              every node/.style={fill=white,circle},
              every edge/.style={draw=red,very thick}]
    \path [-] (1) edge (2);
    \path [-] (3) edge (6);
    \path [-] (7) edge (10);
    \path [-] (8) edge (11);
    \path [-] (12) edge (15);
    \path [-] (16) edge (19);
    \path [-] (20) edge (23);
    \path [-] (1) edge[color=blue,dashed] (4);
    \path [-] (2) edge[color=blue,dashed] (5);
    \path [-] (6) edge[color=blue,dashed] (9);
    \path [-] (10) edge[color=blue,dashed] (13);
    \path [-] (11) edge[color=blue,dashed] (14);
    \path [-] (15) edge[color=blue,dashed] (18);
    \path [-] (22) edge[color=blue,dashed] (25);
    \path [-] (1) edge[color=green,double] (3);
    \path [-] (4) edge[color=green,double] (7);
    \path [-] (5) edge[color=green,double] (8);
    \path [-] (9) edge[color=green,double] (12);
    \path [-] (13) edge[color=green,double] (16);
    \path [-] (14) edge[color=green,double] (17);
    \path [-] (21) edge[color=green,double] (24);
    \path [-] (17) edge (9.5,-2);
    \path [-] (18) edge[color=green,double] (9.5,0);
    \path [-] (19) edge[color=blue,dashed] (9.5,2);
    \path [-] (14) edge[color=green,double] (17);
    \path [-] (21) edge[color=green,double] (24);
    \path [-] (10.3,-2) edge[color=green,double] (10.5,-2);
    \path [-] (10.3,0) edge[color=blue,dashed] (10.5,0);
    \path [-] (10.3,2) edge (10.5,2);
    
\end{scope}

\node at (10,0)[circle,fill,inner sep=0.5pt,color=black]{};
\node at (10.1,0)[circle,fill,inner sep=0.5pt,color=black]{};
\node at (9.9,0)[circle,fill,inner sep=0.5pt,color=black]{};
\node at (10,2)[circle,fill,inner sep=0.5pt,color=black]{};
\node at (10.1,2)[circle,fill,inner sep=0.5pt,color=black]{};
\node at (9.9,2)[circle,fill,inner sep=0.5pt,color=black]{};
\node at (10,-2)[circle,fill,inner sep=0.5pt,color=black]{};
\node at (10.1,-2)[circle,fill,inner sep=0.5pt,color=black]{};
\node at (9.9,-2)[circle,fill,inner sep=0.5pt,color=black]{};

\end{tikzpicture}

\caption{The trident graphs.}
\label{family1}
\end{figure}
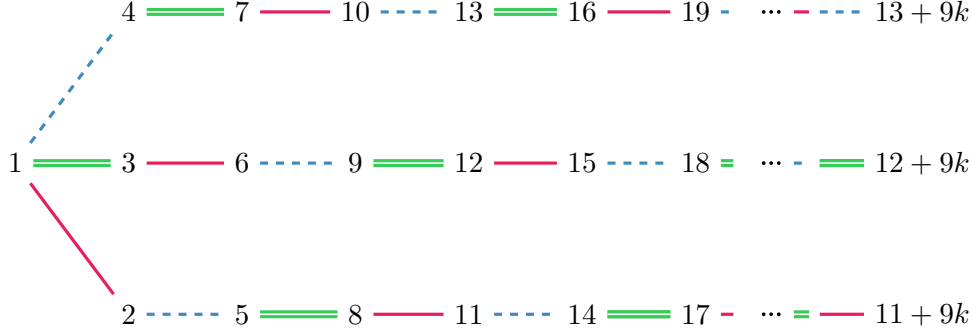

Formally, let $n = 13+9k$ with $k \geq 0$ and define the three involutions for the trident graph $\rho_1,\rho_2,\rho_3 \in \Sv_n$ by

\begin{align*}
    \textcolor{red}{\rho_1} = (1,2)\prod \limits_{i=0}^{k}(8+9i,11+9i) \prod \limits_{i=0}^{k}(3+9i,6+9i) \prod \limits_{i=0}^{k}(7+9i,10+9i). \\
    \textcolor{green}{\rho_2} = (1,3)\prod \limits_{i=0}^{k}(5+9i,8+9i) \prod \limits_{i=0}^{k}(9+9i,12+9i)\prod \limits_{i=0}^{k}(4+9i,7+9i). \\
    \textcolor{blue}{\rho_3} = (1,4)\prod \limits_{i=0}^{k}(2+9i,5+9i) \prod \limits_{i=0}^{k}(6+9i,9+9i)\prod \limits_{i=0}^{k}(10+9i,13+9i).
\end{align*}

As in the previous sections, we define the groups
\begin{align*}
    G = \langle \rho_1, \rho_2 , \rho_3 \rangle , G_1 = \langle \rho_2, \rho_3 \rangle , G_2 = \langle \rho_1, \rho_3 \rangle \text{ and } G_3 = \langle \rho_1, \rho_2 \rangle.
\end{align*}

The goal is to prove the next theorem.

\begin{theorem}\label{goal1}
    For $n=13+9k$ with $k \geq 0$, the coset geometry $\Gamma(G , \{ G_1,G_2,G_3 \})$ is thin, flag transitive, residually connected. Moreover, the induced action by $\Aut(\Gamma) \cong \Alt_n$ if $k$ is even and $\Aut(\Gamma) \cong \Sv_n$ if $k$ is odd, is faithful and $\Gamma(G, \{ G_1,G_2,G_3 \})$ admits trialities but no dualities and the size of the maximal parabolics is constant with $n$.
\end{theorem}

We begin by examining the $G_i$-orbits for $1 \leq i \leq 3$. There are $4$ kind of $G_i$-orbits (see Table \ref{table:orbits}). They are represented in Figure \ref{fig:fam1orbits}.

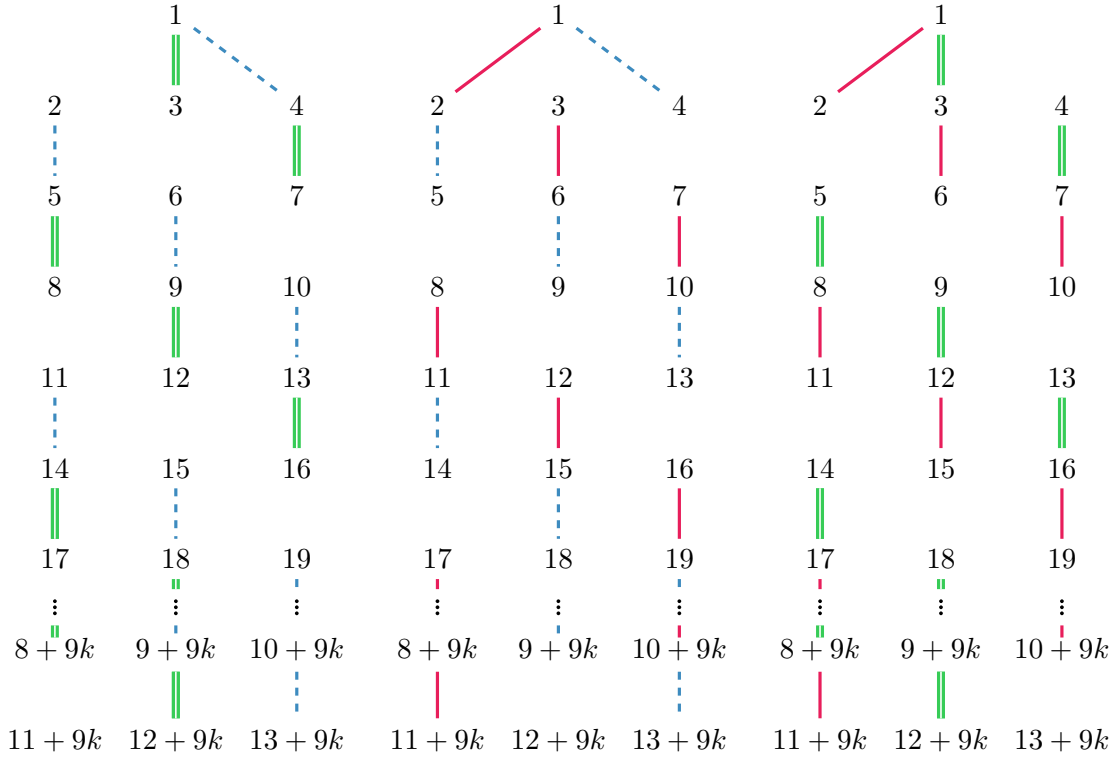
\begin{figure}
\begin{tikzpicture}

\begin{scope}[rotate = 270,scale=0.8,yshift=0]

\begin{scope}[every node/.style={thick}]
    \node (1) at (0,0) {$1$};
    \node (2) at (1.5,-2) {$2$};
    \node (3) at (1.5,0) {$3$};
    \node (4) at (1.5,2) {$4$};
    \node (5) at (3,-2) {$5$};
    \node (6) at (3,0) {$6$} ;
    \node (7) at (3,2) {$7$} ;
    \node (8) at (4.5,-2) {$8$} ;
    \node (9) at (4.5,0) {$9$} ;
    \node (10) at (4.5,2) {$10$} ;
    \node (11) at (6,-2) {$11$} ;
    \node (12) at (6,0) {$12$} ;
    \node (13) at (6,2) {$13$} ;
    \node (14) at (7.5,-2) {$14$} ;
    \node (15) at (7.5,0) {$15$} ;
    \node (16) at (7.5,2) {$16$} ;
    \node (17) at (9,-2) {$17$} ;
    \node (18) at (9,0) {$18$} ;
    \node (19) at (9,2) {$19$} ;
    \node (20) at (10.5,-2) {$8+9k$} ;
    \node (21) at (10.5,0) {$9+9k$} ;
    \node (22) at (10.5,2) {$10+9k$} ;
    \node (23) at (12,-2) {$11+9k$} ;
    \node (24) at (12,0) {$12+9k$} ;
    \node (25) at (12,2) {$13+9k$} ;
\end{scope}

\begin{scope}[>={Stealth[black]},
              every node/.style={fill=white,circle},
              every edge/.style={draw=red,very thick}]
    \path [-] (1) edge[color=blue,dashed] (4);
    \path [-] (2) edge[color=blue,dashed] (5);
    \path [-] (6) edge[color=blue,dashed] (9);
    \path [-] (10) edge[color=blue,dashed] (13);
    \path [-] (11) edge[color=blue,dashed] (14);
    \path [-] (15) edge[color=blue,dashed] (18);
    \path [-] (22) edge[color=blue,dashed] (25);
    \path [-] (1) edge[color=green,double] (3);
    \path [-] (4) edge[color=green,double] (7);
    \path [-] (5) edge[color=green,double] (8);
    \path [-] (9) edge[color=green,double] (12);
    \path [-] (13) edge[color=green,double] (16);
    \path [-] (14) edge[color=green,double] (17);
    \path [-] (21) edge[color=green,double] (24);
    \path [-] (18) edge[color=green,double] (9.5,0);
    \path [-] (19) edge[color=blue,dashed] (9.5,2);
    \path [-] (14) edge[color=green,double] (17);
    \path [-] (21) edge[color=green,double] (24);
    \path [-] (10.1,-2) edge[color=green,double] (10.3,-2);
    \path [-] (10.1,0) edge[color=blue,dashed] (10.3,0);
    
\end{scope}

\node at (9.7,0)[circle,fill,inner sep=0.5pt,color=black]{};
\node at (9.8,0)[circle,fill,inner sep=0.5pt,color=black]{};
\node at (9.9,0)[circle,fill,inner sep=0.5pt,color=black]{};
\node at (9.7,2)[circle,fill,inner sep=0.5pt,color=black]{};
\node at (9.8,2)[circle,fill,inner sep=0.5pt,color=black]{};
\node at (9.9,2)[circle,fill,inner sep=0.5pt,color=black]{};
\node at (9.7,-2)[circle,fill,inner sep=0.5pt,color=black]{};
\node at (9.8,-2)[circle,fill,inner sep=0.5pt,color=black]{};
\node at (9.9,-2)[circle,fill,inner sep=0.5pt,color=black]{};

\end{scope}

\begin{scope}[rotate = 270,scale=0.8,yshift=180]
\begin{scope}[every node/.style={thick}]
    \node (1) at (0,0) {$1$};
    \node (2) at (1.5,-2) {$2$};
    \node (3) at (1.5,0) {$3$};
    \node (4) at (1.5,2) {$4$};
    \node (5) at (3,-2) {$5$};
    \node (6) at (3,0) {$6$} ;
    \node (7) at (3,2) {$7$} ;
    \node (8) at (4.5,-2) {$8$} ;
    \node (9) at (4.5,0) {$9$} ;
    \node (10) at (4.5,2) {$10$} ;
    \node (11) at (6,-2) {$11$} ;
    \node (12) at (6,0) {$12$} ;
    \node (13) at (6,2) {$13$} ;
    \node (14) at (7.5,-2) {$14$} ;
    \node (15) at (7.5,0) {$15$} ;
    \node (16) at (7.5,2) {$16$} ;
    \node (17) at (9,-2) {$17$} ;
    \node (18) at (9,0) {$18$} ;
    \node (19) at (9,2) {$19$} ;
    \node (20) at (10.5,-2) {$8+9k$} ;
    \node (21) at (10.5,0) {$9+9k$} ;
    \node (22) at (10.5,2) {$10+9k$} ;
    \node (23) at (12,-2) {$11+9k$} ;
    \node (24) at (12,0) {$12+9k$} ;
    \node (25) at (12,2) {$13+9k$} ;
\end{scope}

\begin{scope}[>={Stealth[black]},
              every node/.style={fill=white,circle},
              every edge/.style={draw=red,very thick}]
    \path [-] (1) edge (2);
    \path [-] (3) edge (6);
    \path [-] (7) edge (10);
    \path [-] (8) edge (11);
    \path [-] (12) edge (15);
    \path [-] (16) edge (19);
    \path [-] (20) edge (23);
    \path [-] (1) edge[color=blue,dashed] (4);
    \path [-] (2) edge[color=blue,dashed] (5);
    \path [-] (6) edge[color=blue,dashed] (9);
    \path [-] (10) edge[color=blue,dashed] (13);
    \path [-] (11) edge[color=blue,dashed] (14);
    \path [-] (15) edge[color=blue,dashed] (18);
    \path [-] (22) edge[color=blue,dashed] (25);
    \path [-] (17) edge (9.5,-2);
    \path [-] (19) edge[color=blue,dashed] (9.5,2);
    \path [-] (10.1,2) edge (10.3,2);
    \path [-] (10.1,0) edge[color=blue,dashed] (10.3,0);

\end{scope}

\node at (9.7,0)[circle,fill,inner sep=0.5pt,color=black]{};
\node at (9.8,0)[circle,fill,inner sep=0.5pt,color=black]{};
\node at (9.9,0)[circle,fill,inner sep=0.5pt,color=black]{};
\node at (9.7,2)[circle,fill,inner sep=0.5pt,color=black]{};
\node at (9.8,2)[circle,fill,inner sep=0.5pt,color=black]{};
\node at (9.9,2)[circle,fill,inner sep=0.5pt,color=black]{};
\node at (9.7,-2)[circle,fill,inner sep=0.5pt,color=black]{};
\node at (9.8,-2)[circle,fill,inner sep=0.5pt,color=black]{};
\node at (9.9,-2)[circle,fill,inner sep=0.5pt,color=black]{};
\end{scope}

\begin{scope}[rotate=270,scale=0.8,yshift=360]
\begin{scope}[every node/.style={thick}]
    \node (1) at (0,0) {$1$};
    \node (2) at (1.5,-2) {$2$};
    \node (3) at (1.5,0) {$3$};
    \node (4) at (1.5,2) {$4$};
    \node (5) at (3,-2) {$5$};
    \node (6) at (3,0) {$6$} ;
    \node (7) at (3,2) {$7$} ;
    \node (8) at (4.5,-2) {$8$} ;
    \node (9) at (4.5,0) {$9$} ;
    \node (10) at (4.5,2) {$10$} ;
    \node (11) at (6,-2) {$11$} ;
    \node (12) at (6,0) {$12$} ;
    \node (13) at (6,2) {$13$} ;
    \node (14) at (7.5,-2) {$14$} ;
    \node (15) at (7.5,0) {$15$} ;
    \node (16) at (7.5,2) {$16$} ;
    \node (17) at (9,-2) {$17$} ;
    \node (18) at (9,0) {$18$} ;
    \node (19) at (9,2) {$19$} ;
    \node (20) at (10.5,-2) {$8+9k$} ;
    \node (21) at (10.5,0) {$9+9k$} ;
    \node (22) at (10.5,2) {$10+9k$} ;
    \node (23) at (12,-2) {$11+9k$} ;
    \node (24) at (12,0) {$12+9k$} ;
    \node (25) at (12,2) {$13+9k$} ;
\end{scope}

\begin{scope}[>={Stealth[black]},
              every node/.style={fill=white,circle},
              every edge/.style={draw=red,very thick}]
    \path [-] (1) edge (2);
    \path [-] (3) edge (6);
    \path [-] (7) edge (10);
    \path [-] (8) edge (11);
    \path [-] (12) edge (15);
    \path [-] (16) edge (19);
    \path [-] (20) edge (23);
    \path [-] (1) edge[color=green,double] (3);
    \path [-] (4) edge[color=green,double] (7);
    \path [-] (5) edge[color=green,double] (8);
    \path [-] (9) edge[color=green,double] (12);
    \path [-] (13) edge[color=green,double] (16);
    \path [-] (14) edge[color=green,double] (17);
    \path [-] (21) edge[color=green,double] (24);
    \path [-] (18) edge[color=green,double] (9.5,0);
    \path [-] (17) edge (9.5,-2);
    \path [-] (14) edge[color=green,double] (17);
    \path [-] (21) edge[color=green,double] (24);
    \path [-] (10.1,2) edge (10.3,2);
    \path [-] (10.1,-2) edge[color=green,double] (10.3,-2);
\end{scope}

\node at (9.7,0)[circle,fill,inner sep=0.5pt,color=black]{};
\node at (9.8,0)[circle,fill,inner sep=0.5pt,color=black]{};
\node at (9.9,0)[circle,fill,inner sep=0.5pt,color=black]{};
\node at (9.7,2)[circle,fill,inner sep=0.5pt,color=black]{};
\node at (9.8,2)[circle,fill,inner sep=0.5pt,color=black]{};
\node at (9.9,2)[circle,fill,inner sep=0.5pt,color=black]{};
\node at (9.7,-2)[circle,fill,inner sep=0.5pt,color=black]{};
\node at (9.8,-2)[circle,fill,inner sep=0.5pt,color=black]{};
\node at (9.9,-2)[circle,fill,inner sep=0.5pt,color=black]{};
\end{scope}
\end{tikzpicture}

    \caption{$G_i$-orbits for $1 \leq i \leq 3$.}
    \label{fig:fam1orbits}
\end{figure}
    
\begin{table}[h]
\centering
\renewcommand{\arraystretch}{1.3}
\begin{tabular}{c|c|l}
\textbf{Orbit size} & \textbf{Group} & \textbf{Orbits} \\ \hline

4 
& $G_1$ & $\{1,3,4,7\}$ \\
& $G_2$ & $\{1,2,4,5\}$ \\
& $G_3$ & $\{1,2,3,6\}$ \\ \hline

3 
& $G_1$ 
& $\{2+9i,5+9i,8+9i\}$, $\{6+9i,9+9i,12+9i\}$ ($0\le i\le k$);\\[-4pt]
& 
& $\{10+9i,13+9i,16+9i\}$ ($0\le i\le k-1$) \\ \cline{2-3}

& $G_2$ 
& $\{3+9i,6+9i,9+9i\}$, $\{7+9i,10+9i,13+9i\}$ ($0\le i\le k$);\\[-4pt]
& 
& $\{8+9i,11+9i,14+9i\}$ ($0\le i\le k-1$) \\ \cline{2-3}

& $G_3$ 
& $\{4+9i,7+9i,10+9i\}$, $\{5+9i,8+9i,11+9i\}$ ($0\le i\le k$);\\[-4pt]
&
& $\{9+9i,12+9i,15+9i\}$ ($0\le i\le k-1$) \\ \hline

2
& $G_1$ & $\{10+9k,13+9k\}$ \\
& $G_2$ & $\{8+9k,11+9k\}$ \\
& $G_3$ & $\{9+9k,12+9k\}$ \\ \hline

1
& $G_1$ & $\{11+9k\}$ \\
& $G_2$ & $\{12+9k\}$ \\
& $G_3$ & $\{13+9k\}$ \\

\end{tabular}
\caption{$G_i$-orbits for $1 \leq i \leq 3$ sorted by size.}
\label{table:orbits}
\end{table}

\begin{lemma}\label{3action}
    For all $1 \leq i \leq 3$, $G_i = D_{12}$. Moreover, $\tau \in G_i$ is fully determined by its action on the $G_i$-orbit of size $4$ and one of the $G_i$-orbits of size $3$. 
\end{lemma}

\begin{proof}
Without loss of generality we can assume that $i = 1$. There are only $G_1$-orbits of size $4$,$3$ and $2$ and $1$. The action of $\tau$ is the same on every orbit of size $3$ so it suffices to determine the action of $\tau$ on the orbit of size $2$ to fully determine $\tau$. By Lemma \ref{4action}, the action of $\tau \in G_1$ on the unique $G_1$-orbit of size $4$ fully determines the action of $\tau$ on the orbits of size $2$. Since $\text{ord}(\rho_2\rho_3) = 12$, $| G_1 | = 24$.
\end{proof}

The next theorem will help us establish that $G = \langle \rho_1, \rho_2, \rho_3 \rangle = \Alt_n$ or $\Sv_n$. To this end, we make use of a theorem of Jordan, as presented in \cite{12SG}.

\begin{theorem}\label{jordan}
    Let $G \leq \Sv_n$, $n \geq 5$, be a primitive permutation group that contains a $p$-cycle 
    for some prime number $p < n-2$, then $G$ is either the whole symmetric group $\Sv_n$ or the alternating group $\Alt_n$.
\end{theorem}

We start by proving that there exists a $p$-cycle fixing at least $3$ points.

\begin{lemma}
    $G \leq \Sv_n$ contains a $7$-cycle.
\end{lemma}

\begin{proof}
Consider the permutations $(\rho_2\rho_3)^6 \in G_1$, $(\rho_1\rho_3)^6 \in G_2$ and $(\rho_1\rho_2)^6 \in G_3$. They are all $6^{\text{th}}$ power of a permutation in $G_i$ $(1 \leq i \leq 3)$  which means the induced action on the $G_i$-orbits of size $1,2,3$ of these three permutations is trivial. 

This means only the union of the $G_i$-orbits of size $4$ (for $1 \leq i \leq 3$) are potentially not fixed by these permutations. Hence the permutation $(\rho_1\rho_2)^6(\rho_1\rho_3)^6(\rho_2\rho_3)^6$ stabilizes $\{ 1, 2, 3, 4, 5, 6, 7 \}$ and fixes $V\backslash\{ 1,2,3,4,5,6,7\}$.

Straightforward computation shows that 
\begin{itemize}
    \item $\rho_2\rho_3 = (1,7,4,3)(2,5)(6)$.
    \item $\rho_1\rho_3 = (1,4,2,5)(3,6)(7)$.
    \item $\rho_1\rho_2 = (1,6,3,2)(4,7)(5)$.
\end{itemize}
Hence, $(\rho_1\rho_2)^6(\rho_1\rho_3)^6(\rho_2\rho_3)^6 = (1,5,4,6,2,3,7)$ which is a $7$-cycle in $G$.
\end{proof}

It remains to prove that the action of $G$ is primitive to apply theorem \ref{jordan}. 
We prove it in the next lemma by showing that $G$ is doubly transitive on $V$.

\begin{lemma}\label{doublytransitive}
    The action of $G$ is doubly transitive on $V$, and hence primitive.
\end{lemma}

\begin{proof}

Note that any $\PR$-Graph $\mathcal{G} = (V,E,\phi)$ is connected if and only if $G = \langle \rho_1,\rho_2,\rho_3 \rangle$ acts transitively on $V$. Consider the permutations $\textcolor{red}{x_1} = \rho_2 \rho_3 \rho_2 \in G_1$, $\textcolor{green}{x_2} = \rho_3 \rho_1 \rho_3 \in G_2$ and $\textcolor{blue}{x_3} = \rho_1 \rho_2 \rho_1 \in G_3$, all of them in $G_{\{1 \}}$\label{stabi}.

These three involutions induce an edge $3$-coloring of $V = \{ 1, \cdots , n \}$. We can represent them on the graph with their assigned color (see Figure \ref{Primitive1}).

\begin{figure}
    \centering

\begin{tikzpicture}
\tiny
\begin{scope}[every node/.style={thick}]
    \node (1) at (0,0) {$1$};
    \node (2) at (1,-2) {$2$};
    \node (3) at (1,0) {$3$};
    \node (4) at (1,2) {$4$};
    \node (5) at (2,-2) {$5$};
    \node (6) at (2,0) {$6$} ;
    \node (7) at (2,2) {$7$} ;
    \node (8) at (3,-2) {$8$} ;
    \node (9) at (3,0) {$9$} ;
    \node (10) at (3,2) {$10$} ;
    \node (11) at (4,-2) {$11$} ;
    \node (12) at (4,0) {$12$} ;
    \node (13) at (4,2) {$13$} ;
    \node (14) at (5,-2) {$14$} ;
    \node (15) at (5,0) {$15$} ;
    \node (16) at (5,2) {$16$} ;
    \node (17) at (6,-2) {$17$} ;
    \node (18) at (6,0) {$18$} ;
    \node (19) at (6,2) {$19$} ;
    \node (20) at (7,-2) {} ;
    \node (21) at (7,0) {} ;
    \node (22) at (7,2) {} ;
    \node (35) at (8,-2) {} ;
    \node (36) at (8,0) {} ;
    \node (37) at (8,2) {} ;
    \node (23) at (9,-2) {$n-11$} ;
    \node (24) at (9,0) {$n-10$} ;
    \node (25) at (9,2) {$n-9$} ;
    \node (26) at (10,-2) {$n-8$} ;
    \node (27) at (10,0) {$n-7$} ;
    \node (28) at (10,2) {$n-6$} ;
    \node (29) at (11,-2) {$n-5$} ;
    \node (30) at (11,0) {$n-4$} ;
    \node (31) at (11,2) {$n-3$} ;
    \node (32) at (12,-2) {$n-2$} ;
    \node (33) at (12,0) {$n-1$} ;
    \node (34) at (12,2) {$n$} ;
\end{scope}

\begin{scope}[>={Stealth[black]},
              every node/.style={fill=white,circle},
              every edge/.style={draw=red,very thick}]
    \path [-] (2) edge[bend right] (8);
    \path [-] (11) edge[bend right] (17);
    \path [-] (23) edge[bend right] (29);
    \path [-] (6) edge[bend left] (12);
    \path [-] (15) edge[bend right] (6.3,-0.3);;
    \path [-] (27) edge[bend left] (33);
    \path [-] (3) edge[bend left] (7);
    \path [-] (10) edge[bend left] (16);
    \path [-] (31) edge[bend right] (34);
    \path [-] (8.7,1.7) edge[bend right] (28);
    \path [-] (6.3,2.3) edge[bend right] (19);
    \path [-] (24) edge[bend right] (8.7,0.3);

    \path [-] (5) edge[color = green, bend left, double] (4);
    \path [-] (3) edge[bend right,color = green, double] (9);
    \path [-] (8) edge[bend left,color = green, double] (14);
    \path [-] (7) edge[bend left,color = green, double] (13);
    \path [-] (16) edge[bend right,color = green, double] (6.3,1.7);
    \path [-] (17) edge[bend left,color = green, double] (6.3,-1.7);
    \path [-] (29) edge[bend right,color = green, double] (32);
    \path [-] (24) edge[bend right,color = green, double] (30);
    \path [-] (28) edge[bend left,color = green, double] (34);
    \path [-] (25) edge[bend right,color = green, double] (8.7,2.3);
    \path [-] (12) edge[bend right,color = green, double] (18);
    \path [-] (26) edge[bend left,color = green, double] (8.7,-2.3);
    
    \path [-] (2) edge[color = blue, bend left,dashed] (6);
    \path [-] (5) edge[bend left,color = blue,dashed] (11);
    \path [-] (14) edge[bend right,color = blue,dashed] (6.3,-2.3);
    \path [-] (9) edge[bend left,color = blue,dashed] (15);
    \path [-] (18) edge[bend left,color = blue,dashed] (6.3,0.3);
    \path [-] (13) edge[bend right,color = blue,dashed] (19);
    \path [-] (4) edge[bend right,color = blue,dashed] (10);
    \path [-] (25) edge[bend right,color = blue,dashed] (31);
    \path [-] (30) edge[bend right,color = blue,dashed] (33);
    \path [-] (27) edge[bend left,color = blue,dashed] (8.7,-0.3);
    \path [-] (26) edge[bend left,color = blue,dashed] (32);
    \path [-] (23) edge[bend right,color = blue,dashed] (8.7,-1.7);
    
\end{scope}

\node at (7.5,0)[circle,fill,inner sep=0.5pt,color=black]{};
\node at (7.6,0)[circle,fill,inner sep=0.5pt,color=black]{};
\node at (7.4,0)[circle,fill,inner sep=0.5pt,color=black]{};

\node at (7.5,2)[circle,fill,inner sep=0.5pt,color=black]{};
\node at (7.6,2)[circle,fill,inner sep=0.5pt,color=black]{};
\node at (7.4,2)[circle,fill,inner sep=0.5pt,color=black]{};

\node at (7.5,-2)[circle,fill,inner sep=0.5pt,color=black]{};
\node at (7.6,-2)[circle,fill,inner sep=0.5pt,color=black]{};
\node at (7.4,-2)[circle,fill,inner sep=0.5pt,color=black]{};

\end{tikzpicture}

    \caption{$\PR$-Graph on $V$ induced by $\textcolor{red}{x_1},\textcolor{green}{x_2},\textcolor{blue}{x_3}$.}
    \label{Primitive1}
\end{figure}
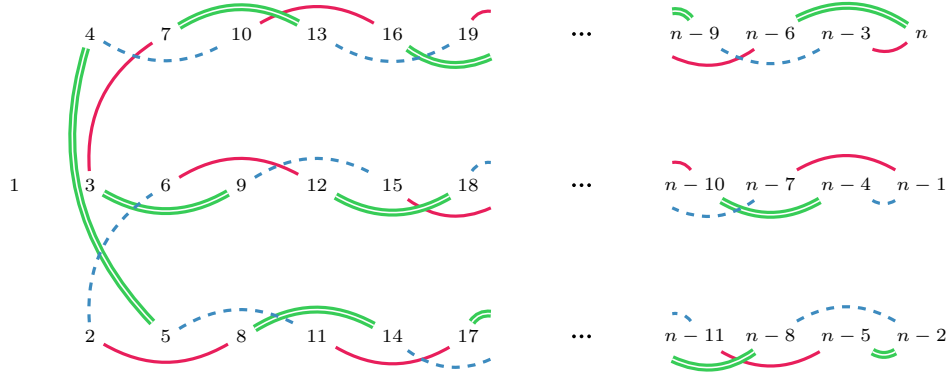

Since this forms a connected graph for $V' = \{ 2 , \cdots , n \}$, $\langle x_1 , x_2 , x_3 \rangle \leq G_{ \{ 1 \}}$ is transitive on $\{ 2 , \cdots , n \}$. Thus the action of $G$ is doubly transitive, hence primitive. 

\end{proof}

\begin{proposition}
    The group $G$ is the alternating group over $n = 16+9k$ elements $\Alt_n$ if $k$ is even and is the symmetric group over $n$ elements $\Sv_n$ if $k$ is odd.
\end{proposition}

\begin{proof}
    
By theorem \ref{jordan}, we conclude $\Alt_n \leq G \leq \Sv_n$. If $k$ is even then $G \leq \Alt_n$ since the involutions $\rho_1,\rho_2,\rho_3$ are products of an even number of two by two disjoint transpositions and $G = \Alt_n$. If $k$ is odd then $G  = \Sv_n$ since for example $\rho_1 \in \Sv_n \backslash \Alt_n$.

\end{proof}

 In this next part, we prove that $G$ acts flag transitively on $\Gamma(G,\{ G_1, G_2, G_3 \})$.

\begin{proposition}\label{FT1}
    $\Gamma(G,\{ G_1, G_2, G_3 \})$ is flag transitive.
\end{proposition}

\begin{proof}
    
Thanks to Lemma \ref{FT}, it is sufficient to prove that $( G_1 \cap G_2 )(G_1 \cap G_3) = G_1 \cap G_2G_3$. Because of lemma \ref{allthin}, we know that $\{ e, \rho_2, \rho_3, \rho_3\rho_2 \} = ( G_1 \cap G_2 )(G_1 \cap G_3) \subseteq G_1 \cap G_2G_3$ is true. Let $\tau \in G_1 \cap G_2G_3$. We want to prove that $\tau \in \{ e, \rho_2, \rho_3, \rho_3\rho_2 \}$.

Thanks to Lemma \ref{3action}, since $\tau \in G_1$, we can list the $24$ elements of $G_1$ according to their action on the $G_1$-orbit of size $4$ ($\{ 1, 3 , 4 , 7 \}$) and one $G_1$-orbit of size $3$, for example $\{ 2 , 5 , 8 \}$, which is enough to reconstruct $\tau$ fully, see Table \ref{complist}.

\begin{table}
    \centering
\begin{tabular}{||c c c||} 
 \hline
 Element $\tau \in G_1$ & Action of $\tau$ on  $\{ 1, 3 , 4 , 7 \}$ & Action of $\tau$ on $\{ 2 , 5 , 8 \}$ \\ [0.5ex] 
 \hline\hline
 $1$ & $(1,4)$ & $(2,5)$\\ 
 \hline
 $2$ & $(1,7,4,3)$ & $(2,8,5)$ \\
 \hline
 $3$ & $(1,3)(4,7)$ & $(5,8)$ \\
 \hline
 $4$ & $(1,3,4,7)$ & $(2,5,8)$ \\
 \hline
 $5$ & $(3,7)$ & $(2,8)$ \\
 \hline
 $6$ & $(1,7,4,3)$ & $e$ \\
 \hline
 $7$ & $(1,7)(3,4)$ & $(2,5)$  \\
 \hline
 $8$ & $(1,4)(3,7)$ & $(2,8,5)$ \\
 \hline
 $9$ & $(1,4)$ & $(5,8)$ \\
 \hline
 $10$ & $e$ & $(2,5,8)$ \\
 \hline
 $11$ & $(1,3)(4,7)$ & $(2,8)$ \\
 \hline
 $12$ & $(3,7)$ & $(5,8)$ \\
 \hline
 $13$ & $(1,4)(3,7)$ & $(2,5,8)$ \\
 \hline
 $14$ & $(1,7)(3,4)$ & $(2,8)$ \\
 \hline
 $15$ & $(1,3,4,7)$ & $e$ \\
 \hline
 $16$ & $(1,3)(4,7)$ & $(2,5)$ \\
 \hline
 $17$ & $e$ & $(2,8,5)$ \\
 \hline
 $18$ & $(1,7)(3,4)$ & $(5,8)$ \\
 \hline
 $19$ & $(1,7,4,3)$ & $(2,5,8)$ \\
 \hline
 $20$ & $(1,4)$ & $(2,8)$ \\
 \hline
 $21$ & $(1,4)(3,7)$ & $e$ \\
 \hline
 $22$ & $(3,7)$ & $(2,5)$ \\
 \hline
 $23$ & $(1,3,4,7)$ & $(2,8,5)$ \\
 \hline
 $24$ & $e$ & $e$ \\
 [0ex] 
 \hline
\end{tabular}
    \caption{The possible actions of $\tau \in G_1$ in $\{ 1,3,4,7 \}$ and $\{ 2,5,8\}$.}
    \label{complist}
\end{table}

The goal is now to justify why all permutations but $\{ e, \rho_2, \rho_3, \rho_3\rho_2 \}$ do not belong to $G_2G_3$. The elements $\rho_3, \rho_2, \rho_3\rho_2, e$ correspond to permutations numbered $1,3,4,24$ respectively.
By Proposition \ref{actionorbit}, the restriction of $\tau$ to $\{1,3,4,7\}$ is equal to one of $\{ e,(1,3)(4,7),(1,4),(1,3,4,7)\}$ and the restriction of $\tau$ on $\{2,5,8\}$ has to be in $\{ e, (2,5),(5,8),(2,5,8)\}$. By removing the permutations not satisfying these conditions from the list, we are left with $8$ permutations (see Table \ref{complist2}).

\begin{table}
    \centering
\begin{tabular}{||c c c||} 
 \hline
 Element $\tau$ of $G_1$ & Action of $\tau$ on  $\{ 1, 3 , 4 , 7 \}$ & Action of $\tau$ on $\{ 2 , 5 , 8 \}$ \\  [0.5ex] 
 \hline\hline
 $1$ & $(1,4)$ & $(2,5)$\\ 
 \hline
 $3$ & $(1,3)(4,7)$ & $(5,8)$ \\
 \hline
 $4$ & $(1,3,4,7)$ & $(2,5,8)$ \\
 \hline
 $9$ & $(1,4)$ & $(5,8)$ \\
 \hline
 $10$ & $e$ & $(2,5,8)$ \\
 \hline
 $15$ & $(1,3,4,7)$ & $e$ \\
 \hline
 $16$ & $(1,3)(4,7)$ & $(2,5)$ \\
 \hline
 $24$ & $e$ & $e$ \\
 [0ex] 
 \hline
\end{tabular}

    \caption{Permutations $\tau \in G_1$ satisfying Proposition \ref{actionorbit}.}
    \label{complist2}
\end{table}

Since $\tau \in G_1 \cap G_2G_3$, $\tau = \tau_2\tau_3$ with $\tau_2 \in G_2$, $\tau_3 \in G_3$.

\begin{enumerate}
    \item If $\tau$ is permutation $9$, since $\tau(4) = 1$ then $\tau_3(4) = 4$ and $\tau_2(4) = 1$. Since $\{ 4,7,10\}$ and $\{ 5,8,11 \}$ are both $G_3$-orbits of size $3$, $\tau_3(4) = 4$ implies that $\tau_3(5)=5$. But, since $\tau(8) = 5$ then $\tau_3(8) = 5$ which is a contradiction.
    \item If $\tau$ is permutation $10$, since $\tau(4) = 4$ then $\tau_3(4)=4$ and $\tau_2(4)=4$. Since $\{ 4,7,10\}$ and $\{ 5,8,11 \}$ are both $G_3$-orbits of size $3$, $\tau_3(5) = 5$. But that would then mean, since $\tau(5) = 8$ that $\tau_2(5) = 8$ which is impossible because $5$ and $8$ belong to different $G_2$-orbits.
    \item If $\tau$ is permutation $15$, since $\tau(5) = 5$ then $\tau_3(5)=5,\tau_2(5)=5$. Moreover, $\tau(7) = 1$ so $\tau_3(7) = 4$ and $\tau_2(4)=1$. Because $\tau_3(7) = 4$ and $\{ 4,7,10\}$ is a $G_3$-orbit of size $3$, the same action of $\tau_3$ needs to happen in the $G_3$-orbit $\{ 5,8,11 \}$ so $\tau_3(8)=5$ which is a contradiction.
    \item If $\tau$ is permutation $16$. Since $\tau(5) = 2$ then $\tau_3(5) = 5$ and $\tau_2(5) = 2$. Moreover, $\tau(7)= 4$ so $\tau_3(7) = 4$ and $\tau_2(4) = 4$. Because $\tau_3(7) = 4$, we get $\tau_3(8) = 5$ for the same reason as before which is a contradiction with $\tau_3(5) = 5$.
\end{enumerate}

The remaining permutations are $\{ e, \rho_2, \rho_3, \rho_3\rho_2 \}$ (see Table \ref{fig:permutationsleft1}). Hence, $\Gamma(G,\{G_1,G_2,G_3\})$ is flag transitive.

\begin{table}
    \centering
\begin{tabular}{||c c c||} 
 \hline
 Element $\tau$ of $G_1$ & Action of $\tau$ on  $\{ 1, 3 , 4 , 7 \}$ & Action of $\tau$ on $\{ 2 , 5 , 8 \}$ \\  [0.5ex]
 \hline\hline
 $1$ & $(1,4)$ & $(2,5)$\\ 
 \hline
 $3$ & $(1,3)(4,7)$ & $(5,8)$ \\
 \hline
 $4$ & $(1,3,4,7)$ & $(2,5,8)$ \\
 \hline
 $24$ & $e$ & $e$ \\
 [0ex] 
 \hline
\end{tabular}
    \caption{Permutations in $\tau \in G_1 \cap G_2G_3$.}
    \label{fig:permutationsleft1}
\end{table}

\end{proof}

Since $\Gamma$ is flag transitive and $\big[ G_1 \cap G_2 \colon G_1 \cap G_2 \cap G_3 \big] = 2$ (see Lemma \ref{allthin} and Lemma \ref{faithful}), $\Gamma$ is thin.

\begin{lemma}
    $\Gamma(G,\{ G_1, G_2, G_3\})$ is residually connected. 
\end{lemma}

\begin{proof}
    We know $G$ acts flag transitively on $\Gamma$ by Lemma \ref{FT1}. Moreover, by Lemma \ref{allthin}, $G_i \cap G_j = \langle \rho_k \rangle$ for any $i,j,k$ such that $\{ i,j,k \} = \{ 1,2,3 \}$. This allows to conclude by Lemma \ref{RC}.
\end{proof}

Since $\Gamma(G,\{ G_1, G_2, G_3\})$ is a flag transitive, residually connected, thin geometry, Proposition \ref{autog} shows that $\Aut(\Gamma) \cong G \cong \Alt_n$ or $\Aut(\Gamma) \cong G \cong \Sv_n$.

We now will verify that $\Gamma$ admits trialities but no dualities. To do so, we use a result from \cite{TrialitySuzuki} which characterizes exactly (in our case) whenever a thin geometry admits certain type-changing correlations.

\begin{proposition}\label{automorphism}
    Let $\Gamma(G,\{ G_1 = \langle \rho_2,\rho_3\rangle, G_2 = \langle\rho_1, \rho_3 \rangle, G_3 = \langle \rho_1,\rho_2\rangle \})$ be a thin, simply flag transitive incidence geometry. Then 
    \begin{enumerate}
        \item $\Gamma$ admits trialities if and only if there exists $\theta \in \Aut(G)$ such that $\theta(\rho_i) = \rho_{i \bmod 3 + 1}$ $(1 \leq i \leq 3)$.
        \item $\Gamma$ admits dualities if and only if there exists $\theta \in \Aut(G)$ such that $\theta(\rho_i) = \rho_{j}, \theta(\rho_j)=\rho_i,\theta(\rho_k) = \rho_k$ for certain $\{ i,j,k \} = \{ 1,2,3 \}$.
    \end{enumerate}
        
\end{proposition}

\begin{proof}
    We prove the case $(1)$, the case $(2)$ is similar. 
    Suppose there exists $\theta \in \Aut(G)$ such that $\theta(\rho_1) = \rho_2, \theta(\rho_2) = \rho_3, \theta(\rho_3) = \rho_1$. Define $\Delta \colon \Gamma \rightarrow \Gamma$ by $\Delta(xG_i) = \theta(x)G_{i \bmod 3 + 1}$. Notice that it is well defined because $\theta \in \Aut(G)$.
    We prove that $\Delta$ is an automorphism of geometries. Indeed, if $x G_i \cap y G_j \neq \emptyset$ then $\theta(xG_i) \cap \theta(yG_j) \neq \emptyset$. Thus $\theta(xG_i)= \theta(x)\theta(G_i) = \theta(x) G_{i \bmod{3} + 1}$ and $\theta(y G_j) = \theta(y) G_{j \bmod{3} + 1}$ and $\Delta(xG_i) \cap \Delta(y G_j) \neq \emptyset$.
    If $\Gamma$ admits trialities, then one can prove there exists $\theta \in \Aut(G)$ cyclically exchanging $\rho_1,\rho_2$ and $\rho_3$. This is proven in \cite{TrialitySuzuki} (proposition $19$).
\end{proof}

Since $\Gamma$ is residually connected, flag transitive and thin, it is simply flag transitive. We can thus apply proposition \ref{automorphism} to determine the existence of trialities and dualities.

\begin{proposition}\label{tri1}
    $\Gamma(G,\{ G_1, G_2, G_3 \})$ admits trialities.
\end{proposition}

\begin{proof}

Let $\tau = \prod \limits_{i=0}^{3+3k}(2+3i,3+3i,4+3i) \in \Alt_n$. 
Conjugation by $\tau$ is an internal automorphism of $\Alt_n$ if $k$ is even (resp. $\Sv_n$ if $k$ is odd) such that $\tau \rho_{i} \tau^{-1} = \rho_{i\bmod 3 +1}$ for $1 \leq i \leq 3$. $\tau$ is better represented with directed edges in Figure \ref{Triality1} by its action on the vertices where it is clear that $\tau \rho_{i \bmod{3} + 1} = \rho_{i \bmod{3} +1} \tau$. 

\end{proof}

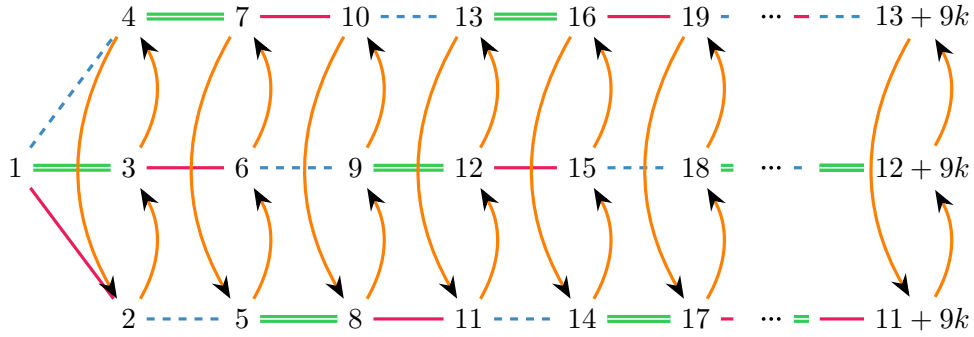
\begin{figure}
    \centering

\begin{tikzpicture}
\begin{scope}[every node/.style={thick}]
    \node (1) at (0,0) {$1$};
    \node (2) at (1.5,-2) {$2$};
    \node (3) at (1.5,0) {$3$};
    \node (4) at (1.5,2) {$4$};
    \node (5) at (3,-2) {$5$};
    \node (6) at (3,0) {$6$} ;
    \node (7) at (3,2) {$7$} ;
    \node (8) at (4.5,-2) {$8$} ;
    \node (9) at (4.5,0) {$9$} ;
    \node (10) at (4.5,2) {$10$} ;
    \node (11) at (6,-2) {$11$} ;
    \node (12) at (6,0) {$12$} ;
    \node (13) at (6,2) {$13$} ;
    \node (14) at (7.5,-2) {$14$} ;
    \node (15) at (7.5,0) {$15$} ;
    \node (16) at (7.5,2) {$16$} ;
    \node (17) at (9,-2) {$17$} ;
    \node (18) at (9,0) {$18$} ;
    \node (19) at (9,2) {$19$} ;
    \node (20) at (10.5,-2) {} ;
    \node (21) at (10.5,0) {} ;
    \node (22) at (10.5,2) {} ;
    \node (23) at (12,-2) {$11+9k$} ;
    \node (24) at (12,0) {$12+9k$} ;
    \node (25) at (12,2) {$13+9k$} ;
\end{scope}

\begin{scope}[>={Stealth[black]},
              every node/.style={fill=white,circle},
              every edge/.style={draw=red,very thick}]
    \path [-] (1) edge (2);
    \path [-] (3) edge (6);
    \path [-] (7) edge (10);
    \path [-] (8) edge (11);
    \path [-] (12) edge (15);
    \path [-] (16) edge (19);
    \path [-] (20) edge (23);
    \path [-] (1) edge[color=blue,dashed] (4);
    \path [-] (2) edge[color=blue,dashed] (5);
    \path [-] (6) edge[color=blue,dashed] (9);
    \path [-] (10) edge[color=blue,dashed] (13);
    \path [-] (11) edge[color=blue,dashed] (14);
    \path [-] (15) edge[color=blue,dashed] (18);
    \path [-] (22) edge[color=blue,dashed] (25);
    \path [-] (1) edge[color=green,double] (3);
    \path [-] (4) edge[color=green,double] (7);
    \path [-] (5) edge[color=green,double] (8);
    \path [-] (9) edge[color=green,double] (12);
    \path [-] (13) edge[color=green,double] (16);
    \path [-] (14) edge[color=green,double] (17);
    \path [-] (21) edge[color=green,double] (24);
    \path [->] (2) edge[bend right, color=orange] (3);
    \path [->] (3) edge[bend right, color=orange] (4);
    \path [->] (4) edge[bend right, color=orange] (2);
    \path [->] (5) edge[bend right, color=orange] (6);
    \path [->] (6) edge[bend right, color=orange] (7);
    \path [->] (7) edge[bend right, color=orange] (5);
    \path [->] (8) edge[bend right, color=orange] (9);
    \path [->] (9) edge[bend right, color=orange] (10);
    \path [->] (10) edge[bend right, color=orange] (8);
    \path [->] (11) edge[bend right, color=orange] (12);
    \path [->] (12) edge[bend right, color=orange] (13);
    \path [->] (13) edge[bend right, color=orange] (11);
    \path [->] (14) edge[bend right, color=orange] (15);
    \path [->] (15) edge[bend right, color=orange] (16);
    \path [->] (16) edge[bend right, color=orange] (14);
    \path [->] (17) edge[bend right, color=orange] (18);
    \path [->] (18) edge[bend right, color=orange] (19);
    \path [->] (19) edge[bend right, color=orange] (17);
    \path [->] (23) edge[bend right, color=orange] (24);
    \path [->] (24) edge[bend right, color=orange] (25);
    \path [->] (25) edge[bend right, color=orange] (23);

    \path [-] (21) edge[color=green,double] (24);
    \path [-] (17) edge (9.5,-2);
    \path [-] (18) edge[color=green,double] (9.5,0);
    \path [-] (19) edge[color=blue,dashed] (9.5,2);
    \path [-] (14) edge[color=green,double] (17);
    \path [-] (21) edge[color=green,double] (24);
    \path [-] (10.3,-2) edge[color=green,double] (10.5,-2);
    \path [-] (10.3,0) edge[color=blue,dashed] (10.5,0);
    \path [-] (10.3,2) edge (10.5,2);
    
\end{scope}

\node at (10,0)[circle,fill,inner sep=0.5pt,color=black]{};
\node at (10.1,0)[circle,fill,inner sep=0.5pt,color=black]{};
\node at (9.9,0)[circle,fill,inner sep=0.5pt,color=black]{};
\node at (10,2)[circle,fill,inner sep=0.5pt,color=black]{};
\node at (10.1,2)[circle,fill,inner sep=0.5pt,color=black]{};
\node at (9.9,2)[circle,fill,inner sep=0.5pt,color=black]{};
\node at (10,-2)[circle,fill,inner sep=0.5pt,color=black]{};
\node at (10.1,-2)[circle,fill,inner sep=0.5pt,color=black]{};
\node at (9.9,-2)[circle,fill,inner sep=0.5pt,color=black]{};

\end{tikzpicture}
    
    \caption{$\tau$ such that $\tau \rho_i \tau^{-1} = \rho_{i \bmod{3} +1}$.}
    \label{Triality1}
\end{figure}

\begin{proposition}
    $\Gamma(G,\{ G_1, G_2, G_3 \})$ does not admit dualities.
\end{proposition}

\begin{proof}
    
We assume without loss of generality, thanks to the triality, that there exists an automorphism $\theta \in \Aut(\Alt_n) \cong \Sv_n$ if $k$ is even (respectively $\theta \in \Aut(\Sv_n) \cong \Sv_n$ if $k$ is odd) such that $\theta(\rho_1) = \rho_2, \theta(\rho_2) = \rho_1$ and $\theta(\rho_3) = \rho_3$. 

As presented in Lemma \ref{doublytransitive}, the permutations $x_1 = \rho_2 \rho_3 \rho_2, x_2 = \rho_3 \rho_1 \rho_3, x_3 = \rho_1 \rho_2 \rho_1$ are such that $X = \langle x_1,x_2,x_3 \rangle$ fixes exactly one element. Nonetheless, the permutations $y_1 = \theta(x_1) = \rho_1 \rho_3 \rho_1, y_2 = \theta(x_2) = \rho_3 \rho_2 \rho_3, y_3 = \theta(x_3) = \rho_2 \rho_1 \rho_2$ should be such that $Y = \theta(X) = \langle y_1,y_2,y_3 \rangle$ also fixes one element. It is easily checked that $Y$ does not fix a single element which is a contradiction. More precisely, for every $v \in V$, there exists $1 \leq i \leq 3$ such that $y_i(v) \neq v$.

\end{proof}

In this section we synthesize some important features of that geometry by computing their Buekenhout diagrams. 

We start by giving their precise definition, which requires to first define some properties of rank $2$ geometries.
 
Francis Buekenhout introduced a diagram associated to incidence geometries ~\cite{buekenhout2013diagram}. We state its definition as in \cite{lineartri}. His idea was to associate to each rank two residue a set of three integers giving information on its incidence graph. 
Let $\Gamma$ be a rank two geometry. We can consider $\Gamma$ to have type set $I = \{P,L\}$, where $P$ and $L$ stand for points and lines. The \textit{point-diameter}, denoted by $d_P(\Gamma) = d_P$, is the largest integer $k$ such that there exists a point $p \in P$ and an element $x \in \Gamma$ with $d(p,x) = k$. Similarly the \textit{line-diameter}, denoted by $d_L(\Gamma) = d_L$, is the largest integer $k$ such that there exists a line $l \in L$ and an element $x \in \Gamma$ with $d(l,x) = k$. Finally, the \textit{gonality} of $\Gamma$, denoted by $g(\Gamma) = g$ is half the length of the smallest circuit in the incidence graph of $\Gamma$.

If a rank two geometry $\Gamma$ has $d_P = d_L = g = n$ for some natural number $n$, we say that it is a \textit{generalized $n$-gon}. Generalized $2$-gons are also called generalized digons. They are in some sense trivial geometries since all points are incident to all lines. Their incidence graphs are complete bipartite graphs. Generalized $3$-gons are projective planes.

 Let $\Gamma$ be a geometry over $I$.  The \textit{Buekenhout diagram} (or diagram for short) $D$ for $\Gamma$ is a graph whose vertex set is $I$. Each edge $\{i,j\}$ is labeled with a collection $D_{ij}$ of rank two geometries. We say that $\Gamma$ belongs to $D$ if every residue of rank two of type $\{i,j\}$ of $\Gamma$ is one of those listed in $D_{ij}$ for every pair of $i \neq j \in I$. In most cases, we use conventions to turn a diagram $D$ into a labeled graph. The most common convention is to not draw an edge between two vertices $i$ and $j$ if all residues of type $\{i,j\}$ are generalized digons, and to label the edge $\{i,j\}$ by a natural integer $n$ if all residues of type $\{i,j\}$ are generalized $n$-gons. It is also common to omit the label when $n=3$.
If the edge $\{i,j\}$ is labeled by a triple $(d_{ij},g_{ij},d_{ji})$ it means that every residue of type $\{i,j\}$ had $d_P = d_{ij}, g = g_{ij}, d_L = d_{ji}$. We can also add information to the vertices of a diagram. 
We can label the vertex $i$ with the number $n_i$ of elements of type $i$ in $\Gamma$. Moreover, if for all flags $F$ of co-type $i$, we have that $|\Gamma_F| = s_i +1$, we will also label the vertex $i$ with the integer $s_i$.

When $k$ is even, since $n = 13+9k$, we have $| G | = \frac{n!}{2}$, $| G_1 | = 24$ so $\big[ G \colon G_1 \big] = \frac{n!}{48}$ and $\big[ G_1 \colon G_1 \cap G_2 \big] = 12$ which is enough to compute the full diagram (see Figure \ref{fig:buekenhoutdiagramfam1}).

When $k$ is odd, $| G | = n!$, $| G_1 | = 24$ so $\big[ G \colon G_1 \big] = \frac{n!}{24}$ and $\big[ G_1 \colon G_1 \cap G_2 \big] = 12$. See Figure \ref{fig:buekenhoutdiagramfam2} for the diagram.

\begin{figure}
\begin{subfigure}{0.05\linewidth}
\begin{tikzpicture}
    
\node (4) at (0,0) {};

\end{tikzpicture}
\end{subfigure}
  \begin{subfigure}[b]{0.4\linewidth}
\begin{tikzpicture}[xscale = 0.9]

\begin{scope}[every node/.style={thick}]
    \draw (0,0) circle (8pt);
    \draw (-2,-2) circle (8pt);
    \draw (2,-2) circle (8pt);
    
    \node (1) at (0,0) {$1$};
    \node (2) at (-2,-2) {$2$};
    \node (3) at (2,-2) {$3$};
    \node (11) at (0,1.1) {$1$};
    \node (12) at (0,0.6) {$\frac{n!}{48}$};
    \node (21) at (-2,-2.5) {$1$};
    \node (22) at (-2,-3) {$\frac{n!}{48}$};
    \node (31) at (2,-2.5) {$1$};
    \node (32) at (2,-3) {$\frac{n!}{48}$};

    \node (4) at (-1.3,-0.7) {$12$};
    \node (5) at (1.3,-0.7) {$12$};
    \node (6) at (0,-2.3) {$12$};
\end{scope}

\begin{scope}[every node/.style={thick}]
    
    \path [-] (-1.7,-2) edge[color=black] (1.7,-2);
    \path [-] (-1.8,-1.8) edge[color=black] (-0.2,-0.2);
    \path [-] (1.8,-1.8) edge[color=black] (0.2,-0.2);
\end{scope}

\end{tikzpicture}

    \caption{Buekenhout diagram for \newline the trident graphs when $k$ is even.}
    \label{fig:buekenhoutdiagramfam1}
\end{subfigure}%
\begin{subfigure}[b]{0.4\linewidth}
\begin{tikzpicture}[xscale = 0.9]

\begin{scope}[every node/.style={thick}]
    \draw (0,0) circle (8pt);
    \draw (-2,-2) circle (8pt);
    \draw (2,-2) circle (8pt);
    
    \node (1) at (0,0) {$1$};
    \node (2) at (-2,-2) {$2$};
    \node (3) at (2,-2) {$3$};
    \node (11) at (0,1.1) {$1$};
    \node (12) at (0,0.6) {$\frac{n!}{24}$};
    \node (21) at (-2,-2.5) {$1$};
    \node (22) at (-2,-3) {$\frac{n!}{24}$};
    \node (31) at (2,-2.5) {$1$};
    \node (32) at (2,-3) {$\frac{n!}{24}$};

    \node (4) at (-1.3,-0.7) {$12$};
    \node (5) at (1.3,-0.7) {$12$};
    \node (6) at (0,-2.3) {$12$};
\end{scope}

\begin{scope}[every node/.style={thick}]
    
    \path [-] (-1.7,-2) edge[color=black] (1.7,-2);
    \path [-] (-1.8,-1.8) edge[color=black] (-0.2,-0.2);
    \path [-] (1.8,-1.8) edge[color=black] (0.2,-0.2);
\end{scope}

\end{tikzpicture}

    \caption{Buekenhout diagram for \newline the trident graphs when $k$ is odd.}
    \label{fig:buekenhoutdiagramfam2}
\end{subfigure}
\caption{Buekenhout diagrams for the trident graphs.}
\end{figure}
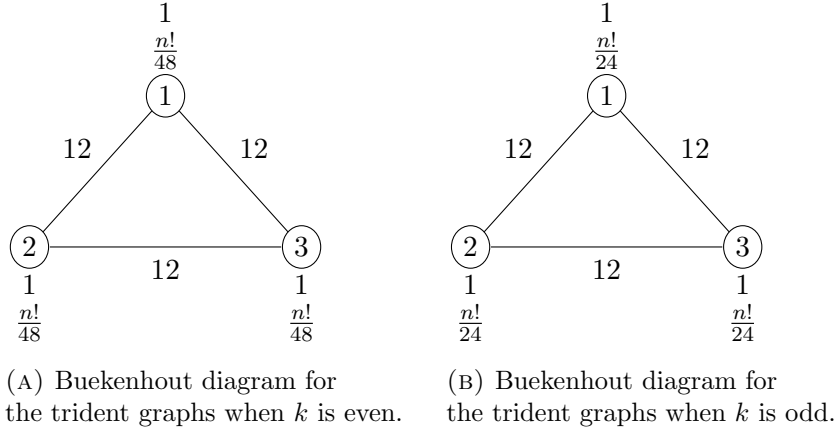

\section{Family 2 : The blooming trident graphs.}

In this section we present another infinite family of thin, residually connected, flag transitive geometries admitting trialities and no dualities with automorphism group $\Sv_n$. This family is fundamentally distinct from the trident graphs, since the size of the maximal parabolics now grows linearly with $n$.
Let $n = 16+6k$ with $k \geq 1$ and define three involutions $\rho_1 , \rho_2 , \rho_3 \in \Sv_n$ as

\begin{align*}
    \textcolor{red}{\rho_1} = (1,2)\prod \limits_{i=0}^{k}(4+6i,7+6i) \prod \limits_{i=0}^{k}(5+6i,8+6i) (9+6k,12+6k)(10+6k,16+6k). \\
    \textcolor{green}{\rho_2} = (1,3)\prod \limits_{i=0}^{k}(2+6i,5+6i) \prod \limits_{i=0}^{k}(6+6i,9+6i)(8+6k,14+6k)(10+6k,13+6k). \\
    \textcolor{blue}{\rho_3} = (1,4)\prod \limits_{i=0}^{k}(3+6i,6+6i) \prod \limits_{i=0}^{k}(7+6i,10+6i)(8+6k,11+6k)(9+6k,15+6k).
\end{align*}

These three involutions form the $\PR$-Graph in Figure \ref{family2}. We refer to these graphs as \textit{blooming trident graphs}. As for the trident graphs, we define 

\begin{align*}
    G = \langle \rho_1, \rho_2 , \rho_3 \rangle , G_1 = \langle \rho_2, \rho_3 \rangle , G_2 = \langle \rho_1, \rho_3 \rangle \text{ and } G_3 = \langle \rho_1, \rho_2 \rangle
\end{align*}

The goal of this section is to prove Theorem \ref{goal2}.

\begin{theorem}\label{goal2}
    For $n=16+6k$ with $k \geq 1$, the coset geometry $\Gamma(G, \{ G_1,G_2,G_3 \})$ is thin, flag transitive and residually connected. Moreover, $\Aut(\Gamma) \cong G \cong \Sv_n$ is faithful, the geometry $\Gamma(G, \{ G_1,G_2,G_3 \})$ admits trialities but no dualities and the size of the maximal parabolics grows essentially linearly with $n$.
\end{theorem}

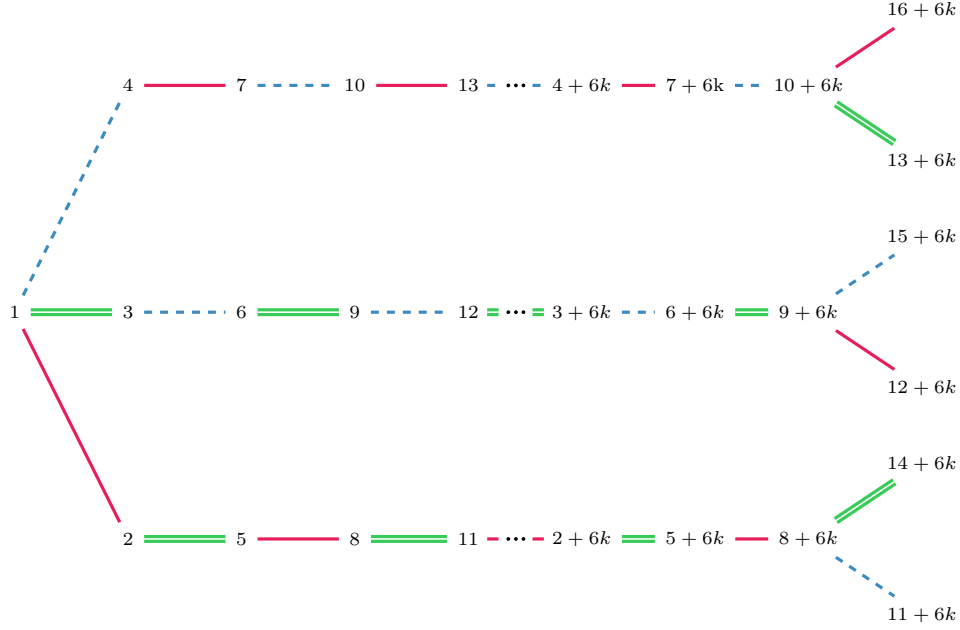
\begin{figure}
    \centering
\begin{tikzpicture}
\begin{scope}[every node/.style={thick,font=\tiny}]
    \node (1) at (0,0) {$1$};
    \node (2) at (1.5,-3) {$2$};
    \node (3) at (1.5,0) {$3$};
    \node (4) at (1.5,3) {$4$};
    \node (5) at (3,-3) {$5$};
    \node (6) at (3,0) {$6$} ;
    \node (7) at (3,3) {$7$} ;
    \node (8) at (4.5,-3) {$8$} ;
    \node (9) at (4.5,0) {$9$} ;
    \node (10) at (4.5,3) {$10$} ;
    \node (11) at (6,-3) {$11$} ;
    \node (12) at (6,0) {$12$} ;
    \node (13) at (6,3) {$13$} ;
    \node (14) at (7.5,-3) {$2+6k$} ;
    \node (15) at (7.5,0) {$3+6k$} ;
    \node (16) at (7.5,3) {$4+6k$} ;
    \node (17) at (9,-3) {$5+6k$} ;
    \node (18) at (9,0) {$6+6k$} ;
    \node (19) at (9,3) {$7+6$k} ;
    \node (20) at (10.5,-3) {$8+6k$} ;
    \node (21) at (10.5,0) {$9+6k$} ;
    \node (22) at (10.5,3) {$10+6k$} ;
    \node (23) at (12,-2) {$14+6k$} ;
    \node (24) at (12,4) {$16+6k$} ;
    \node (25) at (12,-4) {$11+6k$} ;
    \node (26) at (12,-1) {$12+6k$} ;
    \node (27) at (12,2) {$13+6k$} ;
    \node (28) at (12,1) {$15+6k$} ;
    
\end{scope}

\begin{scope}[>={Stealth[black]},
              every node/.style={fill=white,circle},
              every edge/.style={draw=red,very thick}]
    \path [-] (1) edge (2);
    \path [-] (3) edge[color = blue,dashed] (6);
    \path [-] (7) edge[color = blue,dashed] (10);
    \path [-] (8) edge[color=green,double] (11);
    \path [-] (16) edge (19);
    \path [-] (20) edge[color=green,double] (23);
    \path [-] (1) edge[color=blue,dashed] (4);
    \path [-] (2) edge[color=green,double] (5);
    \path [-] (6) edge[color=green,double] (9);
    \path [-] (10) edge (13);
    \path [-] (15) edge[color=blue,dashed] (18);
    \path [-] (1) edge[color=green,double] (3);
    \path [-] (4) edge (7);
    \path [-] (5) edge (8);
    \path [-] (9) edge[color=blue,dashed] (12);
    \path [-] (14) edge[color=green,double] (17);
    \path [-] (17) edge (20);
    \path [-] (18) edge[color=green,double] (21);
    \path [-] (19) edge[color=blue,dashed] (22);
    \path [-] (22) edge (24);
    \path [-] (22) edge[color=green,double] (27);
    \path [-] (21) edge[color=blue,dashed] (28);
    \path [-] (21) edge (26);
    \path [-] (20) edge[color=blue,dashed] (25);
    
    \path [-] (6.25,-3) edge (6.4,-3);
    \path [-] (6.25,0) edge[color=green,double](6.4,0);
    \path [-] (6.25,3) edge[color=blue,dashed] (6.4,3);

    \path [-] (6.85,-3) edge (7,-3);
    \path [-] (6.85,0) edge[color=green,double] (7,0);
    \path [-] (6.85,3) edge[color=blue,dashed] (7,3);
\end{scope}

\node at (6.53,0)[circle,fill,inner sep=0.5pt,color=black]{};
\node at (6.63,0)[circle,fill,inner sep=0.5pt,color=black]{};
\node at (6.73,0)[circle,fill,inner sep=0.5pt,color=black]{};
\node at (6.53,3)[circle,fill,inner sep=0.5pt,color=black]{};
\node at (6.63,3)[circle,fill,inner sep=0.5pt,color=black]{};
\node at (6.73,3)[circle,fill,inner sep=0.5pt,color=black]{};
\node at (6.53,-3)[circle,fill,inner sep=0.5pt,color=black]{};
\node at (6.63,-3)[circle,fill,inner sep=0.5pt,color=black]{};
\node at (6.73,-3)[circle,fill,inner sep=0.5pt,color=black]{};

\end{tikzpicture}

\caption{$\PR$-Graph induced by $\rho_1,\rho_2,\rho_3 \in \Sv_n$.}
\label{family2}
\end{figure}

As for the first family, we begin by computing the size of the maximal parabolics.

\begin{lemma}\label{determinedaction2}
     For each $1 \leq i \leq 3$, we have $| G_i |= 2\lcm(3,6+2k)$. Moreover, any $\rho \in G_i$ is fully determined by its action on the $G_i$-orbit of size $6+2k$ and one of the $G_i$-orbits of size $3$.
\end{lemma}

\begin{proof}
Without loss of generality, we assume $i = 1$. The order $\text{ord}(\rho_2 \rho_3)$ is equal to the least common multiple of the size of the $G_1$-orbits. Hence, in this case, we have $\text{ord}(\rho_2 \rho_3) = \lcm(2,3,6+2k) = \lcm(3,6+2k)$. The integer $6+2k$ being even, knowing the action of $\tau \in G_1$ on the largest $G_1$-orbit is enough to determine its action on the orbits of size $2$ (see Lemma \ref{4action}).
\end{proof}

We now prove that $G$ acts $4$-transitively on $V$. It will help us identify the group $G$. Indeed, because of the non constant sizes of the $G_i$-orbits, the task to find a $p$-cycle as for the trident graphs becomes much harder since the size of the orbits heavily depend on $n$. 
A consequence of the classification of finite simple groups is the classification of $4$-transitive permutation groups. This allows to give another (much more powerful) method for identifying groups on $\PR$-Graphs since the stabilizers of the vertices tend to be large.

\begin{theorem}\label{thm:cameron}
     The finite $2$-transitive groups are explicitly known. In particular, the only finite $6$-transitive groups are symmetric and alternating groups; and the only finite $4$-transitive groups are symmetric, the alternating groups and the Mathieu groups $\Mo,\Md,\Mvd,\Mvt$ and $\Mvq$.
\end{theorem}

\begin{proof}
    See Theorem 4.11 in \cite{cameron}.
\end{proof}

Since $G$ acts on $n = 16+6k$ vertices with $k \geq 1$, {\sc Magma} shows that for $n = 22$, $G = \Sv_{22}$. As $n = | V | > 24$ for $k \geq 2$, if $G$ acts $4$-transitively on $V$ the above theorem implies that $G = \Sv_n$ or $G = \Alt_n$. 

\begin{lemma}\label{family2transitive}
    The action of $G$ is $4$-transitive on $V$.
\end{lemma}

\begin{proof}
    Consider the permutations $x_1 = \rho_3 \rho_2 \rho_3 \in G_1$, $x_2 = \rho_1 \rho_3 \rho_1 \in G_2$ and $x_3 = \rho_2 \rho_1 \rho_2 \in G_3$ all in $G_{\{ 1\}} \cap G_{\{ 8+6k \}} \cap G_{\{ 9 +6k \}} \cap G_{\{ 10 +6k \}}$. The $\PR$-Graph generated by $x_1,x_2,x_3$ is connected in $V \backslash \{ 1, 8+6k, 9+6k, 10+6k \}$. Hence, the action of $G_{ \{1\}}$ is transitive on $V \backslash \{ 1, 8+6k, 9+6k, 10+6k \}$. Let $y_1 = (\rho_2 \rho_3)^2 \rho_1 (\rho_3 \rho_2)^2 \in G_{ \{ 1 \}}$. Whenever $k \geq 1$, we have $y_1(10+6k) = 4+6k \notin \{ 1, 8+6k, 9+6k, 10+6k \}$, and thus $\langle x_1 , x_2 , x_3 , y_1 \rangle \leq G_{ \{ 1\}}$ acts transitively on $V \backslash \{ 1, 8+6k , 9+6k \}$. The two permutations $y_2 = (\rho_3 \rho_1)^2 \rho_2 (\rho_1 \rho_3)^2$ and $y_3 = (\rho_1 \rho_2)^2 \rho_3 (\rho_2 \rho_1)^2 \in G_{ \{ 1 \}}$ serve the same purpose : $y_2(8+6k) = 2+6k\notin \{ 1, 8+6k , 9 +6k,  10+6k \}$ and $y_3(9+6k) = 3+6k \notin \{ 1, 8+6k , 9 +6k,  10+6k \}$. We conclude that 
    $$ \langle x_1 , x_2 , x_3 , y_1 ,  y_2 , y_3 \rangle \leq G_{ \{ 1\}}$$
    acts transitively on $V \backslash\{ 1\}$ hence the action of $G$ on $V$ is doubly transitive.
    Notice that $y_1(8+6k) = 8+6k$, $y_1(9+6k) = 9+6k$, $y_1(10+6k) = 4+6k$, and $y_2(8+6k) = 2+6k$, $y_2(9+6k)=9+6k$, $y_2(10+6k)=10+6k$. Thus, $\langle x_1,x_2,x_3,y_1,y_2 \rangle \leq G_{\{ 1\}} \cap G_{\{ 9+6k \}}$ and acts transitively on $V \backslash \{ 1, 9+6k \}$. Hence, the action is $3$-transitive. Finally $\langle x_1,x_2,x_3,y_1 \rangle \leq G_{\{ 1\}} \cap G_{\{ 8+6k \}} \cap G_{\{ 9+6k \}}$ acts transitively on $V \backslash \{ 1, 8+6k, 9+6k \}$ which means $G$ acts $4$-transitively on $V$ and concludes the proof.
\end{proof}

The argument used in Proposition \ref{family2transitive} when $n = 16$ fails because, among other things, $(\rho_1 \rho_2)^2 \rho_3 (\rho_2 \rho_1)^2 (9) = 10$. As a consequence, $G_{ \{ 1 \}}$ splits $V \backslash \{ 1 \}$ in two orbits, one of them being $\{ 8, 9, 10\}$. Combining the properties established above, we can fully identify the group $G$.

\begin{proposition}\label{autogroup2}
    The group $G$ is the symmetric group over $n$ elements $\Sv_n$.
\end{proposition}

\begin{proof}
    Using Lemma \ref{family2transitive} and Theorem \ref{thm:cameron}, we get that $G = \Alt_n$ or $G = \Sv_n$. Since $\rho_1 \in \Sv_n \backslash \Alt_n$, it must then be that $G = \Sv_n$.
\end{proof}

The next proposition shows that $\Gamma$ is flag transitive. Because of the variable size of maximal parabolics we cannot proceed with a straightforward case listing of the actions of $\tau \in G_1 \cap G_2G_3$ as in Proposition \ref{FT1}. Nevertheless, combining Proposition \ref{actionorbit} and some simplifying arguments, we can reduce the required case analysis to a manageable amount.

\begin{proposition}\label{FT2}
    The geometry $\Gamma(G,\{ G_1, G_2, G_3 \})$ is flag transitive.
\end{proposition}

\begin{proof}

Let $\mathcal{O} = \{4,  1 , 3 , 6 , \cdots , 9 + 6k , 15+6k \}$ be the largest $G_1$-orbit (of size $6+2k$) and $\mathcal{O}_3$ one of the $G_1$-orbits of size $3$. Since $|\mathcal{O}|$ is even, thanks to Lemma \ref{4action}, for $\tau \in G_1$, the knowledge of  $\tau |_\mathcal{O}$ is enough to fully determine $\tau$ on the $G_1$-orbits of size $2$.
To prove that $G_1 \cap G_2G_3 = \{e, \rho_2,\rho_3,\rho_3\rho_2 \}$, we use proposition \ref{actionorbit}. The strategy of the proof will be by exhaustion of cases. We will let the action of $\tau \in G_1 \cap G_2G_3$ on $\mathcal{O}$ be in $\{e, \rho_2,\rho_3,\rho_3\rho_2 \}$ which will determine the action of $\tau$ on the orbit of size $2$ and then let the restriction of $\tau$ to the orbits of size $3$ be different than the one chosen on $\mathcal{O}$ and prove that it leads to a contradiction. There are apriori $12$ different cases to check, listed in Table \ref{fig:family2table}.

\begin{table}
    \centering
    
\begin{tabular}{||c c c||} 
 \hline
 Element $\tau \in G_1 \cap G_2G_3$ & Action of $\tau$ on  $\mathcal{O}$ & Action of $\tau$ on  $\mathcal{O}_3$ \\ [0.5ex] 
 \hline\hline
 $1$ & $e$ & $\rho_2$\\ 
 \hline
 $2$ & $e$ & $\rho_3$ \\
 \hline
 $3$ & $e$ & $\rho_3\rho_2$ \\
 \hline
 $4$ & $\rho_2$ & $e$ \\
 \hline
 $5$ & $\rho_2$ & $\rho_3$ \\
 \hline
 $6$ & $\rho_2$ & $\rho_3\rho_2$ \\
 \hline
 $7$ & $\rho_3$ & $e$  \\
 \hline
 $8$ & $\rho_3$ & $\rho_2$ \\
 \hline
 $9$ & $\rho_3$ & $\rho_3\rho_2$ \\
 \hline
 $10$ & $\rho_3\rho_2$ & $e$ \\
 \hline
 $11$ & $\rho_3\rho_2$ & $\rho_2$ \\
 \hline
 $12$ & $\rho_3\rho_2$ & $\rho_3$ \\
 [0ex] 
 \hline
\end{tabular}

    \caption{List of the possible actions of $\tau \in G_1 \cap G_2G_3$ in $\mathcal{O}$ and $\mathcal{O}_3$ if $\tau \notin \{ e, \rho_3 , \rho_2, \rho_3\rho_2\}$.}
    \label{fig:family2table}
\end{table}

Let $\tau \in G_1 \cap G_2G_3$ be $\tau = \tau_2\tau_3$ with $\tau_2 \in G_2$ and $\tau_3 \in G_3$. Fortunately, we don't need to check all the cases. Indeed, $\tau \rho_2, \rho_3 \tau, \rho_3 \tau \rho_2  \in G_1 \cap G_2G_3$. Assume we prove for example that permutation $1$ cannot exist, then permutation $4$ could not exist as well since by multiplying it on the right by $\rho_2$ we would find permutation $1$. By this line of reasoning, we only need to prove that permutations $1,2$ and $3$ cannot exist.

\begin{enumerate}
    \item Assume $\tau$ is permutation $1$. Then $\tau(10+6k) = 13+6k$. Thus, $\tau_3(10+6k) = 13+6k$ and $\tau_2(13+6k) = 13+6k$. Since $\{ 7+6k , 10+6k , 13+6k \}$ is a $G_3$-orbit as well as $\{ 6+6k , 9+6k , 12+6k\}$, $\tau_3(9+6k) = 6+6k$. But $\tau(9+6k) = 9+6k$ would imply that $\tau_2(6+6k) = 9+6k$ which is a contradiction.
    \item Assume $\tau$ is permutation $2$. Since $\tau(11+6k) = 8+6k$, then $\tau_3(11+6k) = 11+6k$ and $\tau_2(11+6k) = 8+6k$. Since $\{ 5+6k, 8+6k, 11+6k \}$ and $\{ 7+6k, 10+6k, 16+6k \}$ are both $G_2$-orbits of size $3$, $\tau_2(15+6k) = 9+6k$. Since $\tau(15+6k) = 15+6k$, it means $\tau_3(9+6k) = 15+6k$ which is a contradiction.
    \item Assume $\tau$ is permutation $3$. Since $\tau(10+6k) = 13+6k$, $\tau_3(10+6k) = 13+6k$ and $\tau_2(13+6k) = 13+6k$. Thus $\tau_3(9+6k) = 6+6k$. $\tau(9+6k) = 9+6k$ implies that $\tau_2(6+6k) = 9+6k$ which is impossible because $6+6k$ and $9+6k$ belong to different $G_2$-orbits.
\end{enumerate}
\end{proof}

\begin{lemma}
    $\Gamma(G,\{ G_1, G_2, G_3\})$ is residually connected.
\end{lemma}

\begin{proof}
    We know $G$ acts flag transitively on $\Gamma$ by Lemma \ref{FT2}. Moreover, by Lemma \ref{allthin}, $G_i \cap G_j = \langle \rho_k \rangle$ for any $i,j,k$ such that $\{ i,j,k \} = \{ 1,2,3 \}$. This allows to conclude by Lemma \ref{RC}.
\end{proof}

Since $\big[ G_1 \cap G_2 \colon G_1 \cap G_2 \cap G_3 \big] = 2$ (see Theorem \ref{allthin} and Lemma \ref{faithful}) and $\Gamma$ is flag transitive by Proposition \ref{FT2}, it follows that $\Gamma$ is thin. Proposition \ref{autog} and Proposition \ref{autogroup2} then implies that $\Aut(\Gamma) \cong \Sv_n$.

\begin{proposition}
    $\Gamma(G,\{ G_1, G_2, G_3 \})$ admits trialities.
\end{proposition}

\begin{proof}
Similar to Proposition \ref{tri1}. Let $$\tau = (11+6k,12+6k,13+6k)(14+6k,15+6k,16+6k)\prod \limits_{i=0}^{2k+2} (2+3i,3+3i,4+3i).$$
Conjugation by $\tau$ is an automorphism of $\Sv_n$ such that $\tau(\rho_1) = \rho_2, \tau(\rho_2) = \rho_3 , \tau(\rho_3) = \rho_1$. Hence, by Proposition \ref{automorphism}, $\Gamma(G,\{ G_1, G_2, G_3 \})$ admits trialities.

\end{proof}

\begin{proposition}
    $\Gamma(G,\{ G_1, G_2, G_3 \})$ does not admit dualities.
\end{proposition}

\begin{proof}    
We assume without loss of generality, thanks to the triality, that there exists an automorphism $\theta \in \Aut(\Alt_n) \cong \Sv_n$ such that $\theta(\rho_1) = \rho_2, \theta(\rho_2) = \rho_1$ and $\theta(\rho_3) = \rho_3$. 
As in Lemma \ref{family2transitive}, define $x_1 = \rho_3 \rho_2 \rho_3 \in G_1$, $x_2 = \rho_1 \rho_3 \rho_1 \in G_2$ and $x_3 = \rho_2 \rho_1 \rho_2 \in G_3$ These permutations are such that $X = \langle x_1,x_2,x_3 \rangle$ fixes exactly four vertices : $1,8+6k,9+6k$ and $10+6k$. Nonetheless, the permutations $y_1 = \theta(x_1) = \rho_3 \rho_1 \rho_3, y_2 = \theta(x_2) = \rho_2 \rho_3 \rho_2, y_3 = \theta(x_3) = \rho_1 \rho_2 \rho_1$ should be such that $Y = \theta(X) = \langle y_1,y_2,y_3 \rangle$ also fixes four vertices. It is easily checked that $Y$ does not fix a single element which is a contradiction (more precisely for every $v \in V$, $y_i(v) \neq v$ for a certain $1 \leq i \leq 3$).
\end{proof}

Similarly to the case of the trident graphs, we compute the Buekenhout diagram of this geometry.

$| G | = n!$ and $| G_1 | = 2\lcm(3,6+2k)$. If $k = 1 \bmod 3$ or $k = 2 \bmod 3$, $| G_1 | = 2(18+6k) = 36+18k = 3n - 12$. The diagram is Figure \ref{fig:buekenhoutdiagramfam21}.

\begin{figure}
\begin{subfigure}{0.05\linewidth}
\begin{tikzpicture}
    
\node (4) at (0,0) {};

\end{tikzpicture}
\end{subfigure}
  \begin{subfigure}[b]{0.4\linewidth}
\begin{tikzpicture}

\begin{scope}[every node/.style={thick}]
    \draw (0,0) circle (8pt);
    \draw (-2,-2) circle (8pt);
    \draw (2,-2) circle (8pt);
    
    \node (1) at (0,0) {$1$};
    \node (2) at (-2,-2) {$2$};
    \node (3) at (2,-2) {$3$};
    \node (11) at (0,1.1) {$1$};
    \node (12) at (0,0.6) {$\frac{n!}{3n-12}$};
    \node (21) at (-2,-2.5) {$1$};
    \node (22) at (-2,-3) {$\frac{n!}{3n-12}$};
    \node (31) at (2,-2.5) {$1$};
    \node (32) at (2,-3) {$\frac{n!}{3n-12}$};

    \node (4) at (-1.3,-0.7) {$\frac{3n-12}{2}$};
    \node (5) at (1.3,-0.7) {$\frac{3n-12}{2}$};
    \node (6) at (0,-2.3) {$\frac{3n-12}{2}$};
\end{scope}

\begin{scope}[every node/.style={thick}]
    
    \path [-] (-1.7,-2) edge[color=black] (1.7,-2);
    \path [-] (-1.8,-1.8) edge[color=black] (-0.2,-0.2);
    \path [-] (1.8,-1.8) edge[color=black] (0.2,-0.2);
\end{scope}

\end{tikzpicture}
\caption{Buekenhout diagram for \newline the blooming trident graphs \newline when $k = 1,2 \bmod 3$.}
    \label{fig:buekenhoutdiagramfam21}
    \end{subfigure}%
    \begin{subfigure}[b]{0.4\linewidth}
\begin{tikzpicture}

\begin{scope}[every node/.style={thick}]
    \draw (0,0) circle (8pt);
    \draw (-2,-2) circle (8pt);
    \draw (2,-2) circle (8pt);
    
    \node (1) at (0,0) {$1$};
    \node (2) at (-2,-2) {$2$};
    \node (3) at (2,-2) {$3$};
    \node (11) at (0,1.1) {$1$};
    \node (12) at (0,0.6) {$\frac{n!}{2n-28}$};
    \node (21) at (-2,-2.5) {$1$};
    \node (22) at (-2,-3) {$\frac{n!}{2n-28}$};
    \node (31) at (2,-2.5) {$1$};
    \node (32) at (2,-3) {$\frac{n!}{2n-28}$};

    \node (4) at (-1.3,-0.7) {$\frac{2n-28}{2}$};
    \node (5) at (1.3,-0.7) {$\frac{2n-28}{2}$};
    \node (6) at (0,-2.3) {$\frac{2n-28}{2}$};
\end{scope}

\begin{scope}[every node/.style={thick}]
    
    \path [-] (-1.7,-2) edge[color=black] (1.7,-2);
    \path [-] (-1.8,-1.8) edge[color=black] (-0.2,-0.2);
    \path [-] (1.8,-1.8) edge[color=black] (0.2,-0.2);
\end{scope}

\end{tikzpicture}

    \caption{Buekenhout diagram for \newline the blooming trident graphs \newline when $k = 0 \bmod 3$.}
    \label{fig:buekenhoutdiagramfam22}
    \end{subfigure}
    \caption{Buekenhout diagrams for the blooming trident graphs.}
\end{figure}
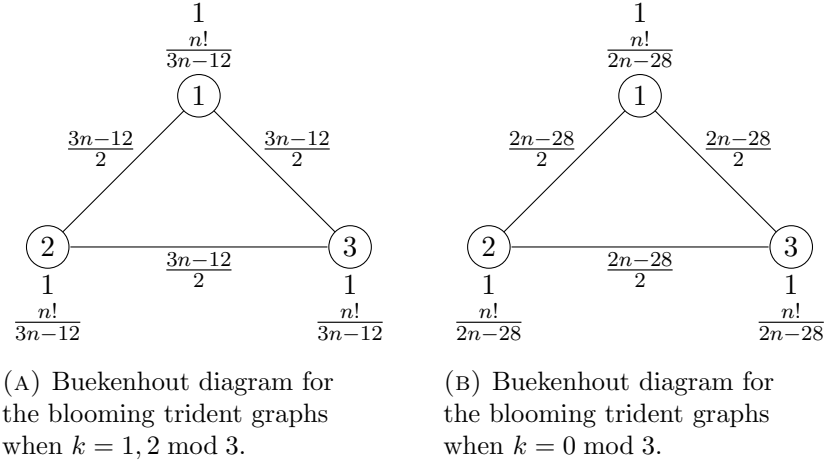

If $k = 0 \bmod 3$, $| G_1 | = 2(2+6k) = 4+12k = 2n - 28$. The diagram is Figure \ref{fig:buekenhoutdiagramfam22}.

\section{Conjectures and Future Directions}

The methods presented in this article introduce new approaches for establishing the existence of thin geometries with specified automorphism and correlation groups (see \cite{representedincidence} for a study of similar existence questions). In this paper, we limited ourselves to rank $3$ geometries (thus requiring $3$ involutions). Nonetheless, similar constructions could be explored for geometries of higher rank involving an arbitrary number of involutions.

The constructions of Families $1$ and $2$ suggest that no essential obstruction prevents the existence of such geometries. We therefore conjecture that they exist for every $n \geq 10$. Computation with {\sc Magma} confirms this conjecture for $10 \leq n \leq 16$ and shows that such geometries do not exist for $n < 10$.

\begin{conjecture}\label{conj:psl}
    Let $n \geq 10$. For $G \in \{\Sv_n, \Alt_n \}$, there exists a flag transitive, residually connected, thin geometry $\Gamma$ such that $\Gamma$ admits trialities and no dualities and $\Aut(\Gamma) \cong G$.
\end{conjecture}

Investigation with {\sc Magma} of tree $\PR$-Graphs suggests that some constructions will almost always produce an induced group $G \in \{ \Sv_n,\Alt_n\}$ with an element in $\Aut(G)$ of order $3$. Let $n = 1+3k$. Consider any proper coloring $\phi$ using colors in $\mathbb{F}_3$ of the path graph $\{ 1,2, \cdots , k+1\}$ where there is an edge between subsequent elements. Construct edges connecting the subsequent vertices in $\{ 1, k+2 , \cdots , 2k+1 \}$ by $\phi(\{ k+i,k+1+i \}) = \phi(\{ i,1+i \}) + 1 \bmod 3$ for $2 \leq i \leq k$ and $\phi(\{ 1,k+2 \}) = \phi(\{ 1,2 \}) + 1 \bmod 3$. Finally, construct the edges connecting the subsequent vertices in $\{ 1, 2k+2 , \cdots , 3k+1 \}$ by $\phi(\{ 2k+i,2k+1+i \}) = \phi(\{ i,1+i \}) + 2 \bmod 3$ for $2 \leq i \leq k$ and $\phi(\{ 1,2k+2 \}) = \phi(\{ 1,2 \}) + 2 \bmod 3$. When $n=7$, the permutations $$\rho_1 = (1,2)(6, 7), \rho_2 = (1,4)(2,3) , \rho_3 = (1,6)(4,5)$$ are such that $\langle \rho_1,\rho_2,\rho_3 \rangle \cong \PSL(2,7)$. When $n=13$, the permutations $$\rho_1 = (1,2)(10,11)(7,8)(12,13), \rho_2 = (1,6)(2,3)(11,12)(4,5), \rho_3=(1,10)(6,7)(3,4)(8,9)$$ are such that $\langle \rho_1,\rho_2,\rho_3 \rangle \cong \PSL(3,3)$, it is represented in Figure \ref{conjecture2}.
Aside from these two exceptional cases, it appears to always be the symmetric or alternating group over $n$ elements. A code generating these graphs randomly is available on the github repository.
\begin{figure}
    \centering
\begin{tikzpicture}
\begin{scope}[every node/.style={thick}]
    \node (1) at (0,0) {$1$};
    \node (2) at (1.5,-2) {$2$};
    \node (6) at (1.5,0) {$6$};
    \node (10) at (1.5,2) {$10$};
    \node (3) at (3,-2) {$3$};
    \node (7) at (3,0) {$7$} ;
    \node (11) at (3,2) {$11$} ;
    \node (4) at (4.5,-2) {$4$} ;
    \node (8) at (4.5,0) {$8$} ;
    \node (12) at (4.5,2) {$12$} ;
    \node (5) at (6,-2) {$5$} ;
    \node (9) at (6,0) {$9$} ;
    \node (13) at (6,2) {$13$} ;
    
\end{scope}

\begin{scope}[>={Stealth[black]},
              every node/.style={fill=white,circle},
              every edge/.style={draw=red,very thick}]
    \path [-] (1) edge (2);
    \path [-] (10) edge (11);
    \path [-] (7) edge (8);
    \path [-] (12) edge (13);
    \path [-] (1) edge[color=blue,dashed] (10);
    \path [-] (3) edge[color=blue,dashed] (4);
    \path [-] (6) edge[color=blue,dashed] (7);
    \path [-] (8) edge[color=blue,dashed] (9);
    \path [-] (1) edge[color=green,double] (6);
    \path [-] (2) edge[color=green,double] (3);
    \path [-] (4) edge[color=green,double] (5);
    \path [-] (11) edge[color=green,double] (12);
    
\end{scope}

\end{tikzpicture}

\caption{A $\PR$-Graph for $\PSL(3,3)$ in Conjecture \ref{conj:psl}.}
\label{conjecture2}
\end{figure}
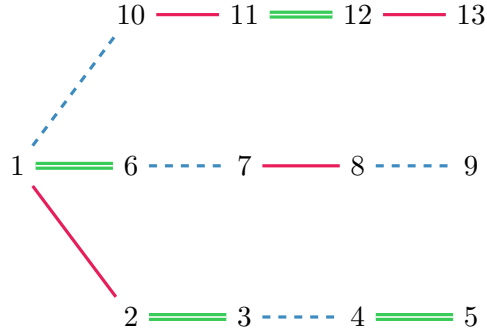

\begin{conjecture}
    The construction described above produces an associated group $G = \langle 
    \rho_1,\rho_2,\rho_3 \rangle$ such that $G \in \{\Alt_n,\Sv_n\}$ except in two cases (up to their coloring symmetries) when $n = 7$ and $n=13$ for which $G$ is isomorphic to $\PSL(2,7)$ and $\PSL(3,3)$ respectively.
\end{conjecture}

The geometry produced for the exceptional case $n=7$ is the triangle complex $\Delta(\PG(2,2))$ described in \cite{lineartri}.

\bibliographystyle{plain} 
\bibliography{refs}

\end{document}